\documentclass[10pt,a4paper]{amsart}

\usepackage{amsfonts}
\usepackage{amsthm}
\usepackage{amssymb}
\usepackage{latexsym}
\usepackage{amsmath}
\usepackage{setspace}
\usepackage{amscd}
\usepackage{verbatim}
\usepackage[pdftex]{graphicx}
\usepackage{graphics}
\usepackage[matrix,arrow]{xy}
\usepackage[active]{srcltx}
\usepackage{mathrsfs}
\usepackage[svgnames]{xcolor}
\usepackage{soul}

\usepackage
[hypertexnames=false,
pdftex,pdfpagemode=UseNone,breaklinks=true,extension=pdf,colorlinks=true,linkcolor=blue,
citecolor=blue,
urlcolor=blue]{hyperref}

\title[Homoclinic orbits, Reeb chords and nice Birkhoff sections]{Homoclinic orbits, Reeb chords and nice Birkhoff sections for Reeb flows in 3D}

\author[V. Colin]{Vincent Colin}
\address{V. Colin, Nantes Universit\'e, CNRS, Laboratoire de Math\'ematiques Jean Leray, LMJL, F-44000 Nantes, France}
\email{vincent.colin@univ-nantes.fr}
\urladdr{https://www.math.sciences.univ-nantes.fr/~vcolin/}

\author[U. Hryniewicz]{Umberto Hryniewicz}
\address{U. Hryniewicz, Chair of Geometry and Analysis, RWTH Aachen,
Pontdriesch 10-12,
52062 Aachen, Germany}
\email{hryniewicz@mathga.rwth-aachen.de}
\urladdr{https://www.mathga.rwth-aachen.de/~hryniewicz/home/}

\author[A. Rechtman]{Ana Rechtman}
\address{A. Rechtman, Institut Fourier,
Universit\'e Grenoble Alpes,
100 rue des math\'ematiques,
38610, Gi\`eres, France\\
Institut Universitaire de France (IUF)}
\email{ana.rechtman@univ-grenoble-alpes.fr}
\urladdr{https://www-fourier.ujf-grenoble.fr/~rechtmaa/}

\date{\today}

\thanks{VC is supported by ANR COSY and the Labex Centre Henri Lebesgue, ANR-11-LABX-0020-01. AR is supported by ANR ANODYN. UH is supported by DFG SFB/TRR 191 ‘Symplectic Structures in Geometry, Algebra and Dynamics’, Projektnummer 281071066-TRR 191.}

\newcommand{\R}{\mathbb{R}}
\newcommand{\Z}{\mathbb{Z}}
\newcommand{\N}{\mathbb{N}}

\theoremstyle{plain}
\newtheorem{theorem}{\sc Theorem}[section]

\newtheorem{proposition}[theorem]{\sc Proposition}
\newtheorem{lemma}[theorem]{\sc Lemma}
\newtheorem{corollary}[theorem]{\sc Corollary}

\newtheorem{definition}[theorem]{\sc Definition}

\theoremstyle{remark}
\newtheorem{remark}[theorem]{\sc Remark}

\newtheorem{question}[theorem]{\sc Question}

\newtheorem{example}[theorem]{\sc Example}

\begin{document}

\maketitle

\begin{abstract}
We prove that for a $C^\infty$-generic contact form defining a given co-oriented contact structure on a closed $3$-manifold, every hyperbolic periodic Reeb orbit admits a transverse homoclinic connection in each of the branches of its stable and unstable manifolds. 
We exploit this result to prove that for a $C^\infty$-generic contact form defining a given co-oriented contact structure, given any finite collection $\Gamma$ of periodic Reeb orbits and any Legendrian link $L$, there exists a global surface of section (embedded Birkhoff section) for the Reeb flow that contains $\Gamma$ in its boundary, and that contains in its interior a Legendrian link that is Legendrian isotopic to $L$ by a $C^0$-small isotopy.
Finally we prove that if the Reeb vector field admits a $\partial$-strong Birkhoff section then every Legendrian knot has infinitely many geometrically distinct Reeb chords, except possibly when the ambient manifold is a lens space or the sphere and the Reeb flow has exactly two periodic orbits.
In particular, $C^\infty$-generically on the contact form there are infinitely many geometrically distinct Reeb chords for every Legendrian knot.
In the case of geodesic flows, every Legendrian knot has infinitely many disjoint chords, without any further assumptions.
\end{abstract}

\tableofcontents

\section{Introduction}

In our previous work~\cite{CDHR}, as well as in the article~\cite{CM} by Contreras and Mazzucchelli, it is proved that a $C^\infty$-generic Reeb vector field on a closed $3$-manifold has a Birkhoff section. This property allows one to transfer results from surface dynamics to Reeb flows. 
As an illustration we obtained in~\cite{CDHR}, making use of techniques from \cite{Mather,Koro2010,KLN,LS}, that $C^\infty$-generically a Reeb vector field has positive topological entropy. In the present paper, we explore two new implications of the existence of a Birkhoff section for a Reeb vector field. 
On the one hand we prove in the presence of a Birkhoff section, as an improvement of the Arnold's chord conjecture \cite{HT1,HT2}, the existence of infinitely many Reeb chords for every Legendrian knot, and a dichotomy finite/infinite for geometrically distinct Reeb chords that is parallel to the famous dichotomy ``two or infinitely many periodic Reeb orbits"~\cite{CDR,CGHHL2}.
In particular, the aforementioned dichotomy for Reeb chords holds $C^\infty$-generically on the contact form as a consequence of~\cite{CDHR}.
On the other hand, we show  the generic existence of homoclinic connections in every pair of stable/unstable manifolds of every hyperbolic orbit.
In return, using this second result, we prove the existence of Birkhoff sections with extra properties, mimicking the open book decomposition case: we find Birkhoff sections that are embedded (global surfaces of section), and even possibly containing a given finite collection of periodic orbits in their boundary and of Legendrian closed curves (modulo $C^0$-small Legendrian isotopy) in their interior. We detail our results in the three subsections below.

Through the paper, all objects considered are assumed to be smooth, unless explicitly stated otherwise. Given a co-oriented contact structure~$\xi$ of a closed oriented 3-manifold $M$, we say that a contact form~$\lambda$ defines~$\xi$ if $\xi=\ker\lambda$ and if $\lambda$ defines the given co-orientation. We say that $R$ is a Reeb vector field for $\xi$ if it is the Reeb vector field of a contact form $\lambda$ that defines $\xi$, i.e. is the unique solution of the equations $\lambda(R)=1$ and $d\lambda(R,\cdot)=0$.
Note that the map that sends a contact form that defines $\xi$ to its Reeb vector field gives a diffeomorphism between the set of contact forms defining $\xi$ and the set of Reeb vector fields for $\xi$, when both sets are equipped with the $C^\infty$ topology.
A property of Reeb vector fields for $\xi$ is said to be $C^\infty$-generic, or to hold $C^\infty$-generically, if the set of contact forms that define~$\xi$ and whose  Reeb vector field satisfies the given property contains a $G_\delta$-set for the $C^\infty$-topology.

\subsection{Homoclinic intersections}

A deep result due to Katok \cite{katok} implies that, for a flow on a closed 3-manifold, the topological entropy is positive if, and only if, there exists a hyperbolic periodic orbit whose stable and unstable manifolds have a transverse intersection. 
An orbit contained in such an intersection is called a transverse homoclinic orbit of the corresponding hyperbolic periodic orbit. As mentioned before, positive topological entropy is $C^\infty$-generic among the Reeb vector fields of a contact structure~\cite{CDHR}. 
In the first part of this paper 
we go one step further and prove that, for a  $C^\infty$-generic Reeb vector field, every hyperbolic periodic orbit has a transverse homoclinic orbit.

\begin{theorem}
\label{thm: homoclinic}
Let $\xi$ be a co-oriented contact structure on a closed 3-manifold $M$. 
The set of Reeb vector fields for $\xi$ such that every hyperbolic periodic orbit has a transverse homoclinic orbit in each branch of its stable and unstable manifolds is $C^\infty$-generic.
\end{theorem}

The genericity hypotheses are explained below; see Theorem~\ref{thm: explicit}. 
In particular, they imply that hyperbolic periodic orbits are dense in~$M$.
We refer the reader to Sections~\ref{section: homoclinic} and \ref{sec: geodesic}.

Transverse homoclinics for every hyperbolic periodic orbit of a 3D flow was proved, in the case of Reeb vector fields, by Contreras and Oliveira under extra hypotheses on either the genus of a Birkhoff section, or properties of a supporting broken book decomposition.
These hypotheses are not known to be generic. 
We refer to~\cite[Theorem~B]{CO_geod} for details. 

Contreras and Oliveira also proved the existence of transverse homoclinic orbits for every hyperbolic periodic orbit of a $C^\infty$-generic geodesic flow of a surface, see~\cite[Theorem~A]{CO_geod}. 
Geodesic flows are Reeb, but observe that the difficulty in this case is that one needs genericity hypotheses on the Riemannian metric on the surface and not only on the flow on the unitary tangent bundle. 
One of the genericity hypotheses used in Theorem~\ref{thm: homoclinic} is not generic among geodesic flows, nonetheless we can still give another proof of the theorem by Contreras and Oliveira, see Section~\ref{sec: geodesic}.

\begin{theorem}[Contreras and Oliveira~\cite{CO_geod}]
\label{thm: geodesic}
The set of Riemannian metrics on a closed orientable surface such that every hyperbolic periodic orbit has a transverse homoclinic orbit in each of the branches of its stable/unstable manifolds is $C^\infty$-generic among Riemannian metrics.
\end{theorem}

The proof of Theorem~\ref{thm: homoclinic} is mainly based on the existence of a Birkhoff section and on the techniques used by Le Calvez and Sambarino to prove: 

\begin{theorem}[Le Calvez, Sambarino \cite{LS}]\label{thm: LS}
Let $S$ be a closed surface of genus~$g$, and $f$ be a strongly non-degenerate area preserving diffeomorphism of class $C^r$, with $r\geq 1$. If $f$ satisfies Zehnder's condition on elliptic periodic points, and if the number of hyperbolic periodic points of $f$ is strictly greater than $\max(0, 2g-2)$, then every hyperbolic periodic orbit has a transverse homoclinic orbit.
\end{theorem}

Zehnder's condition is explained below. A diffeomorphism $f$ is non-degenerate if for every $k\in \mathbb{N}$ the differential of $f^{kn}$ at a periodic point of period $n$ does not admits one as an eigenvalue. This implies that periodic points are either hyperbolic (when the eigenvalues are real) or elliptic (when the differential of $f^n$ is an irrational rotation). In this case, $f$ is strongly non-degenerate if the intersection between the stable and unstable manifolds of the hyperbolic periodic points are transverse. 

Theorem \ref{thm: LS} dictates sufficient hypotheses  to deduce the existence of tranverse homoclinic orbits in the case of flows in~3D. 
First of all, the flow has to be strongly non-degenerate. 
In analogy to diffeomorphisms, a Reeb flow is non-degenerate if no root of unity belongs to the spectrum of the linearized Poincar\'e map of a periodic orbit.
The periodic orbits of a non-degenerate Reeb vector field are either elliptic, positive hyperbolic or negative hyperbolic, if the eigenvalues are non-real, positive real or negative real, respectively. 
Also, the intersections of the stable and unstable manifolds of the hyperbolic periodic orbits of a strongly non-degenerate flow are transverse.  

There are two other genericity hypotheses on Reeb vector fields that we will use. 
Set $\lambda$ to be the contact form defining the Reeb vector field. We need periodic orbits to be equidistributed with respect to the invariant volume $\lambda\wedge d\lambda$, in the sense that there are finite collections of periodic orbits (with weights) such that the invariant probability measure supported on them approaches arbitrarily well the normalized invariant volume form in the weak* topology. We refer to the normalized invariant volume as the Liouville measure. Irie \cite{irie} proved that this condition is $C^\infty$-generic among Reeb vector fields for a given co-oriented contact structure. 

The other condition we need is on elliptic periodic orbits, to which we refer as Zehnder's condition, since it was studied by Zehnder in~\cite{Zeh}. A diffeomorphism $f$ of a surface satisfies Zehnder's condition if for every elliptic periodic point $z$ and for every $U$ neighborhood of $z$, there exists a topological closed disc $D\subset U$ containing $z$, whose boundary is made of finitely many pieces of stable and unstable manifolds of some hyperbolic periodic point $z'$. For flows, we ask this condition to be satisfied for the Poincar\'e map of every elliptic periodic orbit. This condition is $C^\infty$-generic by Zehnder's results applied to the Poincar\'e maps of elliptic periodic orbits; see also the discussion in Section~6.3 of~\cite{CDHR}.

For a diffeomorphism $f$ of a surface, the dynamics near a generic elliptic fixed point is known to be very rich. When $f$ is real analytic, Birkhoff established a normal form for $f$ and deduced that generically the elliptic fixed point can be approximated by a sequence of periodic points containing both elliptic and hyperbolic periodic points. KAM theorem asserts that there are many invariant concentric circles around the elliptic fixed point, some of which bound {\it instability regions} that are annuli where one finds hyperbolic periodic points having homoclinic orbits (we refer for example to Section 20 of \cite{AA}). This behavior exists in every neighborhood of a generic elliptic fixed point.

Theorem~\ref{thm: homoclinic} is a direct consequence of the following statement:

\begin{theorem}\label{thm: explicit}
Let $R$ be a Reeb vector field 
on a closed 3-manifold such that:
\begin{itemize}
\item[(G1)] it is strongly non-degenerate;
\item[(G2)] its periodic orbits are equidistributed with respect to the Liouville measure;
\item[(G3)] it satisfies Zehnder's condition on elliptic periodic orbits.
\end{itemize}
Then every hyperbolic periodic orbit has a transverse homoclinic orbit in each of the branches of its stable/unstable manifolds. 
\end{theorem}

As mentioned above, (G1), (G2) and (G3) are $C^\infty$-generic among Reeb vector fields. Condition (G2)  is not known to be generic among geodesic flows. Observe that:
\begin{itemize}
    \item condition (G1) implies the existence of a Birkhoff section~\cite{CM}. The same conclusion follows for a non-degenerate Reeb flow satisfying (G2)~\cite{CDHR}. We denote by (G2+) the condition non-degenerate and (G2).
    \item conditions (G2) and (G3) together imply that hyperbolic periodic orbits are dense, and in particular infinite. 
\end{itemize}

\begin{remark}
According to Contreras and Oliveira \cite{CO_primeends}, Theorem~\ref{thm: LS} can be proved without Zehnder's condition. Theis statement implies that one can then remove it from the hypotheses in Theorem~\ref{thm: explicit}.
\end{remark}

\bigskip

Let us sketch the proof of Theorem~\ref{thm: explicit}. Since the Reeb vector field is strongly non-degenerate and satisfies Irie's equidistribution of periodic orbits~\cite{irie}, it has infinitely many periodic orbits, and by Theorem~1.1 in~\cite{CDHR} or by Theorem~A in~\cite{CM}, it admits a $\partial$-strong Birkhoff section ${S}$. Collapsing the boundary components of ${S}$ we obtain a closed surface $\overline{S}$. The first return map to the Birkhoff section gives a homeomorphism $f$ of $\overline{S}$, that is $C^\infty$ and non-degenerate except along a finite collection of periodic points coming from the blow-down of the boundary components of $S$. One can then compute the Lefschetz indices of the degenerate periodic points in this collection.
We first show that modulo considering another Birkhoff section, one can always obtain $f$ without Lefschetz index zero periodic points -- this step uses the equidistribution of periodic orbits (G2) and more precisely Proposition~2.11 and Theorem~2.10 of \cite{CDHR}. The new homeomorphism has infinitely many hyperbolic periodic orbits (or more precisely, negative Lefschetz index periodic points).
If $f$ was non-degenerate on the whole $\overline{S}$, Theorem~\ref{thm: LS} would imply Theorem \ref{thm: explicit} directly. Instead, one has to revisit certain proofs from~\cite{LS} to address the degenerate points.

\medskip

To obtain a proof of Theorem~\ref{thm: geodesic}, first we observe that conditions (G1) and (G3) are $C^\infty$-generic among geodesic flows. For  Zehnder's condition, we provide an explanation of this fact in Section~\ref{sec: geodesic}, based on previous works. Equidistribution of periodic orbits, that is condition (G2), is not known to be generic among geodesic flows. Instead we use that in the case of a geodesic flow of a closed  surface $\Sigma$, Irie \cite{Irie_0} proved that $C^\infty$-generically on the Riemannian metric, closed geodesics are dense on $\Sigma$. We use this property to get rid of Lefschetz index zero points by adding to the original Birkhoff section some Birkhoff annuli over closed geodesics. 

\bigskip

\subsection{Improved Birkhoff sections}

The second part of the paper
explores how Theorem~\ref{thm: homoclinic} allows to improve the construction of Birkhoff sections from \cite{CDHR}. 
We look for embedded Birkhoff sections, known in the literature as global surfaces of section, containing a given finite collection of periodic orbits in the boundary. 
This echoes the statement proved by Giroux asserting that for every link $\Lambda$ transverse to a contact structure there exists a supporting open book decomposition containing $\Lambda$ in its binding. 
Inspired by the open book decomposition case, we also ask:

\begin{question}
\label{question: Legendrian} 
For a contact form $\lambda$ on a closed $3$-manifold whose Reeb vector field has infinitely many periodic orbits and a Legendrian link~$L$ for $\xi = \ker\lambda$, can one find a Birkhoff section for the Reeb flow that contains~$L$?
\end{question}

Note that when a Reeb vector field has finitely many periodic orbits, a Legendrian that crosses all of them will never be contained in the interior of a Birkhoff section nor in a global surface of section, since it would have to exit along its boundary. Moreover, it is easy to find a Legendrian in the irrationnal ellipsoid with two periodic Reeb orbits that is sufficiently knotted so that it is not contained in any Birkhoff section (and not only in the interior).

We answer positively Question~\ref{question: Legendrian} if~the contact form is allowed to be $C^\infty$-generic and~$L$ is allowed to be moved by a $C^0$-small Legendrian isotopy. 
At the same time we can impose any given finite set of periodic orbits to be part of the boundary as well as embeddedness of the Birkhoff section.
We refer the reader to Section~\ref{sec: BSconstrained}.

\begin{theorem}\label{thm: BS1}
Let $R$ be a Reeb vector field satisfying the hypotheses (G1), (G2) and (G3).
Given a finite collection $\Gamma$ of periodic orbits and a Legendrian link $L$, there exists a Legendrian link $L'$ that is Legendrian isotopic to $L$ by a $C^0$-small isotopy, and a $\partial$-strong Birkhoff section $S$ for $R$ such that
\begin{enumerate}
    \item $\Gamma\subset \partial S$;
    \item $S$ is embedded;
    \item $S$ contains $L'$ in its interior.
\end{enumerate}
\end{theorem}

\begin{remark}
\label{remark: tangent} 
Since the Birkhoff section $S$ is transverse to $R$, it can be deformed so that $TS|_{L'}=\xi\vert_{L'}$.
\end{remark}

We have hence that the these properties of Birkhoff sections are $C^\infty$-generic among Reeb vector fields. 

\begin{remark}
Parts (1) and (2) of Theorem~\ref{thm: BS1} can be proved for a Reeb vector field satisfying hypothesis (G1), (G3) and such that every hyperbolic periodic orbit has a homoclinic in each of the stable and unstable branches.
Since these three hypotheses
 are generic among geodesic flows, (1) and (2) also hold for generic geodesic flows. 
\end{remark}

The proof of Theorem~\ref{thm: BS1} is presented here by steps. We first prove that given a collection of periodic orbits $\Gamma$ there is a Birkhoff section containing it in its boundary. This step uses Zehnder's condition (G3) to add elliptic periodic orbits to the binding and the conclusion of Theorem~\ref{thm: explicit} to add hyperbolic periodic orbit to the binding (see Section~\ref{sec: constraint}). To be more precise, the homoclinic orbits allow to construct Fried's pair of pants incorporating a given hyperbolic orbit to the binding. We proceed to make such a Birkhoff section embedded by adding new orbits to the binding. This step uses once more the generic hypothesis (G3) and the existence of homoclinics (see Section~\ref{sec: embedded}). The ideas in these two parts are quite similar. 

To finish the proof of Theorem~\ref{thm: BS1}, the third part deals with finding a Birkhoff section that contains a given Legendrian (allowing to move the Legendrian by small Legendrian isotopy). Once more, the proof relies on Theorem \ref{thm: BS1} and is obtained by successively adding Fried's pair of pants sections to a given embedded Birkhoff section to gradually 
improve the relative position of a given Legendrian link, in particular getting rid of all its intersections with the section and of its Reeb chords that do not intersect the section (see Section~\ref{sec: legendrian}).

\subsection{Infinitely many Reeb chords}

In the third part of the paper, corresponding to Section~\ref{sec: chords},
we explore what can be said about Reeb chords of a Legendrian knot assuming the existence of a Birkhoff section. 
Since the resolution of the Arnold's chord conjecture in dimension three by Hutchings and Taubes \cite{HT1,HT2}, it is known that every Legendrian knot $K$ contained in a closed contact $3$-manifold $(M,\lambda)$ has a Reeb chord, i.e. a non constant interval piece of Reeb trajectory whose endpoints lie on $K$. We say that two Reeb chords are geometrically distinct if, as subsets of the ambient manifold $M$, they are different. Observe that two different Reeb chords can be geometrically the same if they are contained in a periodic orbit. 
Here we pass from the existence of one Reeb chord to infinitely many in the presence of a Birkhoff section.

\begin{theorem}\label{thm: chords}
If a Reeb vector field for a co-oriented contact structure $\xi$ on a closed connected $3$-manifold has a Birkhoff section, then every Legendrian knot $L$ has infinitely many Reeb chords. 
 
If the Birkhoff section is $\partial$-strong and if $L$ has finitely many geometrically distinct Reeb chords, then the Birkhoff section is a disk or an annulus and the Reeb vector field has exactly two periodic orbits. 
Moreover, there are at least two geometrically distinct Reeb chords.
\end{theorem}

Observe that this result gives a Reeb chord version of the ``two or infinitely many  Reeb periodic orbits" theorem proven in \cite{CDR} for non-degenerate Reeb vector fields and in \cite{CGHHL2} in the general possibly degenerate case under the hypothesis that the first Chern class of the contact structure is torsion.
For a proof we refer the reader to Section~\ref{sec: chords}. 

In the first part of Theorem~\ref{thm: chords}, the infinite number of chords could come from a periodic orbit intersecting the Legendrian knot in one point, that gives different, but not disjoint, chords obtained by covering the periodic orbit several times.
The case of a Reeb vector field with exactly two periodic orbits was studied extensively in \cite{CGHHL} where it is proved that the first return map on the Birkhoff section is an irrational pseudo-rotation, the two periodic orbits, together with their multiples, are elliptic non-degenerate, and their actions are related to the volume of the manifold as in the case of irrational ellipsoids.
In this case, the ambient manifold $M$ is either a lens space or the 3-sphere.

The proof of Theorem \ref{thm: chords} relies on the existence of at least one Reeb chord provided by Hutchings and Taubes \cite{HT1,HT2}, together with the fact that the first return map on a Birkhoff section has flux zero in the case of a Reeb flow (we refer to Definition~\ref{defn: fluxzero}). This theorem is independent of the rest of the results in this paper. 

\medskip

As mentioned before, under one of the  $C^\infty$-generic conditions on the contact form: (G1) the Reeb field is strongly non-degenerate, or (G2+) the Reeb vector field is non-degenerate and the corresponding Liouville measure is approximated by periodic orbits;
the Reeb vector field admits a $\partial$-strong Birkhoff section (we refer to \cite{CM} and \cite{CDHR}). 
We thus obtain:

\begin{corollary}
If $\lambda$ is a contact form on a closed connected $3$-manifold $M$ that satisfies one of the $C^\infty$-generic conditions (G1) or (G2+),
then every Legendrian knot $L$ in $(M,\xi=\ker\lambda)$ has infinitely many Reeb chords. 
Moreover, if the Reeb flow has infinitely many periodic orbits, then $L$ has infinitely many geometrically distinct Reeb chords. 
\end{corollary}

Regarding the fact that, on one hand by a result of Alves and Mazzuchelli~\cite{AM} every geodesic flow on the unit tangent bundle of a Riemannian surface has a Birkhoff section which is $\partial$-strong, and on the other hand infinitely many periodic orbits, we also get:

\begin{corollary}
If $(\Sigma, g)$ is a closed Riemannian surface, then every Legendrian knot in $(UT\Sigma,\lambda_g)$ has infinitely many geometrically distinct Reeb (geodesic flow) chords.
\end{corollary}

Now, if we also allow a $C^\infty$-small perturbation of the Legendrian knot, Theorem~\ref{thm: homoclinic} combined with the fact that the hyperbolic periodic orbits of a Reeb vector field are generically dense (conditions (G2) and (G3) together) implies:

\begin{corollary}\label{cor: chords} 
For a Reeb vector field on a closed contact $3$-manifold satisfying the $C^\infty$-generic conditions that hyperbolic periodic orbits are dense and all have transverse homoclinic connections, every Legendrian knot can be deformed by a $C^\infty$-small Legendrian isotopy to have an exponentially growing number of honest Reeb chords with respect to the action.
\end{corollary}
Here, a Reeb chord is {\it honest} if it is not contained in a periodic orbit.

\begin{question}\label{question: infinite}  Does a Legendrian knot in a closed contact $3$-manifold $M$ always have infinitely many Reeb chords?
Does it have infinitely many geometrically distinct chords as soon as $M$ is not a lens space or the sphere endowed with a Reeb vector field with exactly two periodic orbits?
\end{question}

In order to delimitate Question \ref{question: infinite}, we provide an example with only four  geometrically distinct chords, showing that we cannot expect in general infinitely many disjoint Reeb chords.

\begin{example}
For $\theta>1$ an irrationnal number, we consider the contact form $\lambda_\theta=(\theta+1)dy +t(\theta dx-dy)$ on $(\R/\Z)^2_{x,y}\times [0,1]_t=T^2\times [0,1]$. The Reeb vector field is tangent to the tori $T^2\times \{*\}$ and has constant irrationnal slope $\theta$ in the $(x,y)$ coordinates. The slope of the characteristic foliation of $T^2\times \{t\}$ ranges from $0$ when $t=0$ to $-1$ when $t=1$, both slopes are strictly smaller than $\theta$. We can thus blow-down $T^2\times \{0\}$ and $T^2\times \{1\}$ along their characteristic foliations to get a sphere $S^3$. The contact form $\lambda$ survives in $S^3$ and its Reeb flow has two periodic orbits. The Legendrian arcs $\{x=0,y=0\} \cup \{x=0, y=1/2\}$ and $\{x=1/2,y=0\} \cup \{x=1/2, y=1/2\}$ all together form a closed Legendrian knot $L$ in $S^3$ that has all its chords contained in the periodic orbits. The surfaces $A_0=\{x=0\}$ and $A_{1/2}=\{x=1/2\}$ are two Birkhoff sections, with one positive and one negative boundary component. Together they form a smooth $2$-torus with non singular characteristic foliation. The Legendrian knot $L$ is contained in it and intersects the binding $\partial A_0$ at $4$ points in between where the $6$  geometrically distinct chords stand.
On the other hand, the periodic orbits provide infinitely many different Reeb chords. 
\end{example}

\subsection{Organization of the paper}

We do not aim for a self-contained text since some technical issues were addressed in \cite{CDHR}; we give precise references for each part. We do start recalling definitions and concepts that appear in the proof of Theorem~\ref{thm: explicit}, in  Section~\ref{sec: prem}. The proof of Theorem~\ref{thm: explicit} is given in Section~\ref{section: homoclinic} and the adaptation of the proof to the case of geodesic flows is explained in Section~\ref{sec: geodesic}. The proof of Theorem~\ref{thm: BS1} is contained in Section~\ref{sec: BSconstrained}, separated into three different parts as explained above. The proof of Theorem~\ref{thm: chords} and its corollaries is contained in Section~\ref{sec: chords}. Finally, in the appendix we explain a refined construction of Fried pair of pants that is used in the proof of Theorem~\ref{thm: BS1}.

\subsection{Acknowledgements}

We wholeheartedly thank Pierre Dehornoy for numerous discussions all throughout this work and for his suggestions regarding the existence of embedded Birkhoff sections (in particular Appendix~\ref{appendix}). We also thank Thomas Barthelm\'e for some discussions that motivated the results in Section~\ref{sec: constraint}.

\section{Preliminaries}\label{sec: prem}

Let $(M,R)$ be a closed 3-manifold endowed with a non-singular vector field $R$ and denote by $\phi^t$ its flow. 
A {\it section} for $R$ is the image under an immersion $\iota$ of a compact surface $S$ with boundary such that: $\iota(\partial S)$ is a collection of periodic orbits of $R$ and the interior of $S$ is embedded and transverse to $R$. 
For simplicity, we sometimes denote the immersed surface by $S$.

\begin{definition}
A section $\iota(S)$ is a  Birkhoff section if it intersects the orbit of $R$ starting at every point of $M$ in bounded time. 
\end{definition}

\begin{remark}
The definition of a Birkhoff section does not assume that it is connected. If a Birkhoff section is not connected, then each connected component is again a Birkhoff section. Hence we assume that Birkhoff sections are connected.
\end{remark}

Observe that a periodic orbit in $\iota(\partial S)$ can be covered several times by the image of a connected component of $\partial S$ and also be in the image of different connected components of $\partial S$. Each periodic orbit $\gamma$ in $\iota(\partial S)$ has an integer number attached, the {\it local multiplicity}, that is the number of times the image of a connected component of $\partial S$ covers  $\gamma$ with orientations. That is, the interior of  $S$ is oriented so that it is positively transverse to the vector field $R$ and hence the periodic orbits in $\partial S$ have two orientations: the one induced by $S$ and the one given by $R$. If the two orientations coincide the local multiplicity is positive, and if not it is negative. Since the interior of $S$ is embedded, the local multiplicity is the same for every connected component of $\partial S$ that is mapped to a periodic orbit $\gamma$. For a periodic orbit $\gamma \subset \partial S$ we define its multiplicity to be sum of the local multiplicities of the boundary components of $S$ that cover $\gamma$. Observe that this is equal to the number of boundary components of $S$ that cover $\gamma$ multiplied by the local multiplicity of one of them.

A Birkhoff section is a global surface of section (GSS for short) if in addition the surface is embedded, or equivalently if the multiplicities of the periodic orbits in the boundary are all equal to $\pm 1$.

Let $S$ be a Birkhoff section for a Reeb vector field $R=R_\lambda$. The first return map $h:S\to S$ is a homeomorphism that preserves the exact area form $d\lambda$. Denote by $h_*$ the induced map on $H_1(S;\mathbb{Z})$. The map $h$ satisfies also that it is \emph{flux zero}:

\begin{definition}\label{defn: fluxzero}
A homeomorphism $h:S\to S$ has {\it flux zero} if for every curve $\gamma$ whose homology class is in $\ker (h_*-I)$, we have that ~$\gamma$ and $h(\gamma)$ cobound a $d\lambda$-area zero $2$-chain in $S$.
\end{definition}

If $\gamma$ is a periodic orbit  then $E_\gamma = TM|_\gamma/T\gamma$ is a vector bundle over~$\gamma$. 
Hence $\mathbb{P}^+\gamma = (E_\gamma \setminus 0)/\R_+ \to \gamma$ is a circle bundle. Then the linearised flow $D\phi^t|_\gamma$ determines a smooth flow on $\mathbb{P}^+\gamma$.
We call this flow the linearized flow on $\mathbb{P}^+\gamma$, with no fear of ambiguity. Blowing-up $\gamma$ means replacing $\gamma$ by $\mathbb{P}^+\gamma$, thus obtaining a 3-manifold with boundary $M_\gamma$, whose boundary is exactly one torus. The smooth vector field $R$ defines a smooth vector field in $M_\gamma$ that we keep denoting by $R$. 

If $\iota:S\to M$ is a section for $R$ and $c$ is a connected component of $\partial S$ such that $\iota(c) =\gamma$, then $\iota$ defines a smooth map $\nu^c_\iota:c \to \mathbb{P}^+\gamma$ as follows: choose any smooth vector field $n$ of $S$ along $c$ pointing outwards, so that $d\iota(n)$ is a map $c \to TM|_\gamma$, and define $\nu^c_\iota$ to be the map obtained by composing $d\iota(n)$ with the quotient map $TM|_\gamma \to \mathbb{P}^+\gamma$.
The definition of $\nu^c_\iota$ does not depend on the choice of $n$.
The trace of $\iota$ along $c$ is defined as the image of the map $\nu^c_\iota$, in particular it is a subset of $\mathbb{P}^+\gamma$.

\begin{definition}
Let $\iota:S\to M$ be a section for $R$, we say that $\iota$ is a $\partial$-strong section if its trace along every connected component $c \subset \partial S$ is an embedding transverse to the linearized flow on $\mathbb{P}^+\gamma$ (here $\gamma$ is the periodic orbit in $\iota(\partial S)$ that contains $\iota(c)$).
 \end{definition}
 
In this paper we study Reeb vector fields, that form a particular family of volume preserving non-singular vector fields. Let $M$ be a closed oriented 3-manifold, $\xi$ a co-oriented contact structure and $\lambda$ a positive contact form defining $\xi$, that is $\xi=\ker \lambda$ and $\lambda\wedge d\lambda> 0$. The Reeb vector field $R_\lambda$ of $\lambda$ is defined by the equations $$ \iota_{R_\lambda}d\lambda=0 \qquad \mbox{and} \qquad \lambda(R_\lambda)=1. $$
Observe that $R_\lambda$ preserves the volume form $\lambda\wedge d\lambda$ as well as the plane field $\xi$. Thus, it is a non-singular volume preserving vector field on $M$. When there is no need to specify the contact form we will write $R$ for the Reeb vector field.
 
We can now state a theorem that we will use.

\begin{theorem}\label{thm: existenceBS}
Let $R$ be a non-degenerate Reeb vector field on a closed 3-manifold. If $R$ satisfies one of the following two conditions:
\begin{itemize}
\item $R$ is strongly non-degenerate;
\item the periodic orbits of $R$ are equidistributed with respect to the invariant volume form;
\end{itemize}
then $R$ admits a $\partial$-strong Birkhoff section.
\end{theorem}

The conclusion with the first hypothesis is due to Contreras and Mazzucchelli \cite{CM} (except to the $\partial$-strong part that can be achieved), and with the second hypothesis is due to the authors \cite{CDHR}.

\subsection{Definitions from topological surface dynamics}\label{sec: topological}

Here we review results on the dynamics of surface homeomorphisms due to Mather~\cite{Mather}, Koropecki~\cite{Koro2010}, Koropecki, Le Calvez and Nassiri~\cite{KLN}, and Le Calvez and Sambarino~\cite{LS}, as well as a couple of basic notions in Carath\'eodory's theory of prime ends. We also recover results from \cite{CDHR}, to which we refer for a longer discussion.

In this section, $S$ is an open oriented surface without boundary.
A boundary representative of $S$ is a sequence $P_1 \supset P_2 \supset P_3 \dots$ of  connected, not relatively compact, open sets such that $\partial P_i$ is compact for every $i$, and  such that for every $K \subset S$ compact there exists $n_K$ such that $i>n_K$ implies $P_i \cap K = \emptyset$.
Here $\partial P_i$ denotes the topological boundary of $P_i$ relative to $S$.
Two boundary representatives $\{P_i\}$ and $\{P'_i\}$ are equivalent if for every $n$ there exists $m$ such that $P_m \subset P'_n$, and \textit{vice versa}.
An ideal boundary point or end is an equivalence class of boundary representatives.
The set of ideal boundary points of~$S$, also called the ideal boundary of~$S$, is denoted by $b_IS$.
The ideal completion of $S$ is defined by $c_IS = S \sqcup b_IS$, with the topology generated by the open subsets of $S$, and sets of the form $V \cup V'$ where $V$ is an open subset of $S$ whose boundary relatively to $S$ is compact, and $V' \subset b_IS$ consists of points that admit a boundary representative $\{P_i\}$ satisfying $P_i \subset V´$ for every $i$.

Let $U \subset S$ be open. 
The impression of $p \in b_IU$ in $S$ is the set 
$$
Z(p) = \bigcap_{\substack{V \subset c_IU \ \text{open} \\ p\in V}} {\rm cl}_S(V\cap U)
$$
where ${\rm cl}_S$ denotes closure relative to $S$. Note that $Z(p)$ is closed in $S$ and contained in the topological boundary $\partial U$ of $U$ relative to $S$.
If $S$ has finite genus then an ideal boundary point $p\in b_IU$ is said to be regular if $p$ is isolated in $b_IU$ and $Z(p)$ has more than one point.

A compact set $K \subset S$ is called a continuum if it is connected and has at least two points. 
A continuum $K$ is said to be cellular if $K = \cap_{n\in \N}D_n$ where each $D_n \subset S$ is a closed disk, and $D_{n+1}$ is contained in the interior of $D_n$, for every~$n$.
Cellular continua are contractible in $S$.
A continuum $K$ is said to be annular if $K = \cap_{n\in \N}A_n$, where each $A_n \subset S$ is homeomorphic to a closed annulus, $A_{n+1}$ is contained in the interior of $A_n$, and $K$ separates both boundary components of $A_n$, for every $n$.

From now on, $S$ is assumed to have finite genus.
Let $U \subset S$ be open and $p\in b_IU$ be regular.
In Carath\'eodory's theory of prime ends one constructs a space $b_{\mathscr{E}_p}U$ homeomorphic to a circle, in such a way that if $U$ is invariant by an orientation preserving homeomorphism $f:S\to S$ then there is an induced orientation preserving homeomorphism on $b_{\mathscr{E}_p}U$. Its Poincar\'e's rotation number is denoted by $\rho(f,p) \in \R/\Z$.
The reader is referred to~\cite[Section~3]{KLN} for a nice introduction to the theory of prime ends.

\begin{definition}\label{defn: regulardomain}
A connected open set $V\subset S$ is a regular domain if $V$ has finitely many ends and its complement has finitely many connected components such that none of them is an isolated point.
\end{definition}

A regular domain $V\subset S$ admits three different compactifications:
\begin{itemize}
\item[-] the closure $\overline V$ of $V$ in $S$;
\item[-] the ideal boundary or end compactification $c_IV$, observe that $c_IV$ is a closed surface without boundary;
\item[-] the prime end compactification $c_{\mathscr{E}_p}V=V\cup b_{\mathscr{E}_p}V$ that can also be obtained by blowing-up each point in $b_IV$ to a circle.
\end{itemize}
An orientation preserving homeomorphism $f$ of $V$ that preserves a measure that is positive on non-empty open sets (implying that $f$ is $\partial$-non wandering in the sense of \cite[Definition 3.10]{KLN}), admits extensions to $c_IV$ and to $c_{\mathscr{E}_p} V$ that we denote respectively  $f_I$ and $f_{\mathscr{E}_p}$. 

We can now state a useful result from \cite{LS}:

\begin{corollary}[Corollary 2.3 from \cite{LS}]\label{coro: 2.3LS}
Let $f:S\to S$ be a diffeomorphism  preserving a measure that is positive on non-empty open sets of a closed surface $S$, such that $f$ is strongly non-degenerate and satisfies Zehnder's condition on elliptic periodic points. Let $V$ be a regular domain invariant by $f$. If $p\in b_IV$ is an end of $V$ of period $q$, then for every $n\geq 1$, the Lefschetz index of $f_I^{nq}$ at $p$ is equal to 1. 
\end{corollary}

Observe that this result can be applied to periodic regular domains of $f$ replacing $f$ by a suitable iteration.

\medskip

We now review the notions of Moser stable and Mather sectorial periodic points from~\cite[Section~5]{Mather}. Let $f:S\to S$ be an orientation preserving homeomorphism and $p$ be a fixed point of $f$.

Choose a contracting homeomorphism $\alpha$ of $[0,+\infty)$, i.e. $\alpha^n(t) \to 0$ for all $t\geq0$.
The map $(s,t) \mapsto (\alpha(s),\alpha^{-1}(t))$ defines a homeomorphism of $[0,+\infty) \times [0,+\infty)$ denoted by $\alpha \times \alpha^{-1}$.
An elementary sector for $(f,p)$ is a closed subset $U \subset S$ such that $p \in \partial U$, $q$ has a neighborhood $N$ in $U$ satisfying $f(N) \subset U$ and $f^{-1}(N) \subset U$, and the germ of $f|_U$ at $p$ is topologically conjugate to the germ of $\alpha \times \alpha^{-1}$ at $(0,0)$.
Then $p$ is said to be a Mather sectorial fixed point of $f$ if a neighborhood of $p$ in $S$ is a finite union of elementary sectors for $(f,p)$.
It turns out that this definition is independent of the choice of $\alpha$. If the number of elementary sectors in $2\ell$, we say that $p$ is $\ell$-pronged. A non-degenerate hyperbolic fixed point is Mather sectorial and 2-pronged.

The point $p$ is a Moser stable fixed point of $f$ if every neighborhood of $p$ in $S$ contains an $f$-invariant disk $D$ with $p$ in its interior, such that $f|_{\partial D}$ has an orbit dense in $\partial D$. After Mather's work \cite{Mather}, in the case of non-degenerate elliptic periodic points, the hypothesis of having an invariant disk neighborhood was replaced in other works by Zehnder's generic condition. The results of Mather were reproved under the later condition in \cite{KLN}.

If $p$ is a periodic point of $f$, then $p$ is a Mather sectorial periodic point of $f$ if it is a Mather sectorial fixed point of $f^n$ for some $n$. 
Similarly, $p$ is a Moser stable periodic point of $f$ if it is a Moser stable fixed point of $f^n$ for some $n$.

\section{Homoclinic orbits of Reeb vector fields}
\label{section: homoclinic}

In this section we prove Theorem~\ref{thm: explicit}, that implies Theorem~\ref{thm: homoclinic} as discussed in the introduction. Consider a Reeb vector field $R$ with contact form $\lambda$ on a closed 3-manifold $M$, satisfying the hypothesis of Theorem~\ref{thm: explicit}. Then $R$ admits a $\partial$-strong Birkhoff section by Theorem~\ref{thm: existenceBS}, that we denote by $\iota:S\to M$. Let $L$ be the link of periodic orbits $\iota(\partial S)$.

\subsection{Periodic points of the first return map}\label{sec: firstreturn}

Our goal in this section is to describe the type of periodic points of the first return map to the Birkhoff section $S$ for  a vector field satisfying properties (G1), (G2) and (G3). 

We start by blowing up the link $L = \iota(\partial S)$, that is we replace each connected component of $\gamma\in L$  by $\mathbb{P}^+\gamma$ and obtain a $3$-manifold $M_L$ with boundary, whose boundary components are invariant tori. As before, the vector field $R$ defines a vector field in $M_L$ that we keep denoting $R$.

There is a unique immersion $\widehat\iota:S \to M_L$ that agrees with the original immersion of $S$ on $\dot S = S\setminus\partial S$, maps $\partial S$ to $\partial M_L$ and is transverse to $\partial M_L$ along $\partial S$.
Moreover, as explained in \cite{CDHR},  $\widehat\iota$ defines a smooth embedding $S\hookrightarrow M_L$ that defines a global section for the extended flow on $M_L$.
The return map on~$\widehat\iota(S)$ is a smooth diffeomorphism.

For simplicity of notation, we will denote by $S$ the immersed surface in $M$ and the embedded surface in $M_L$. We now consider the ideal completition of $\dot S$, that is the space $c_I\dot S$ that is  the closed orientable surface obtained from $S$ by collapsing each boundary component to a point. According to Definition~\ref{defn: regulardomain}, $\dot S$ is not a regular open set in $c_I\dot S$.

The return map on~$S$ induces an orientation preserving homeomorphism 
$$
f : c_I\dot S \to c_I\dot S
$$
that is smooth on $\dot S$, and such that $b_I\dot S$ consists of finitely many periodic points, that we denote by $p_1,p_2,\ldots, p_k$. Moreover, this map preserves a measure $\mu$ that is positive on non-empty open sets and agrees on $\dot S$ with a constant multiple of the measure induced by $d\lambda$.

We now discuss the type of periodic points of $f$ and their Lefschetz index as fixed points of iterations of $f$. For a periodic point $p\in c_I\dot S$, we denote by $q$ or $q(p)$ the minimal period.

\medskip

\noindent {\bf Periodic points in $\dot S$.}\\
 Every periodic point in $\dot S$ comes from a periodic orbit of the  flow in $M\setminus L$ that is either elliptic or hyperbolic. The set of hyperbolic periodic points decomposes further into two sets: positive and negative hyperbolic periodic points depending on the sign of the eigenvalues of the differential of $f^q$. Let $p\in \dot S$ be a periodic point of minimum period $q$, the Lefschetz index of $f^{nq}$, for $n\geq 1$, at $p$ is: 
\begin{itemize}
\item equal to 1 if $p$ is elliptic;
\item equal to $(-1)^{n+1}$ if $p$ is negative hyperbolic;
\item equal to $-1$ if $p$ is positive hyperbolic.
\end{itemize}

\noindent{\bf The periodic points $p_i$ for $1\leq i\leq k$.}\\
If $p_i$ corresponds to a  component $\gamma$ of $L$ that is a hyperbolic periodic orbit of the  flow, the point $p_i$ is  Mather sectorial. If $\gamma$ is a positive hyperbolic periodic orbit and the minimal period of $p_i$ is $q_i$, taking an arbitrarily small tubular neighborhood of $\gamma$ in $M$,  the surface $S\subset M$ intersects  the stable manifold of $\gamma$ along $\ell(p_i)$ segments and the unstable manifold of $\gamma$ along $\ell(p_i)$ segments having an endpoint in $p_i$. Since $\gamma$ is positive hyperbolic, the number $\ell(p_i)$ is even. Thus in $c_I\dot S$ we have two intertwined collections of $\ell(p_i)$-segments  that are invariant. We say that $p_i$ is $\ell(p_i)$-pronged, the case $\ell(p_i)=2$ being as a non-degenerate positive hyperbolic periodic point. We denote these two collections of curves by $W^*(p_i)$, with $*=s,u$, respectively.

The dynamics of $f^{q_i}$ on these collections of curves is the following. There is a circular order on each family, that depends on whether the orientation of $\gamma$ as a boundary component of $S$ coincides or not with the orientation of $\gamma$ as orbit of the flow. Then $f^{q_i}$ maps a curve in one of the families to the second next curve in the same family, according to this order.  Then $f^{(q_i\ell(p_i))/2}$ is the first iteration of $f$ that fixes $p_i$ and each curve of the two families. The Lefschetz index of $p_i$ is
$$
ind(f^{nq_i}, p_i)=\left\{\begin{matrix} 1 & \mbox{if}\,\,\, n\not\equiv 0 \mod{\ell(p_i)/2}\\  1-\ell(p_i) & \mbox{if}\,\, \, n \equiv 0 \mod{\ell(p_i)/2}\end{matrix}\right.
$$
We say that $p_i$ is a $\ell(p_i)$-pronged positive Mather sectorial periodic point of $f$.

If $\gamma\subset L$ is a negative hyperbolic periodic orbit,  the surface $S\subset M$ intersects  the stable manifold of $\gamma$ along $\ell(p_i)$ segments and the unstable manifold of $\gamma$ along $\ell(p_i)$ segments in any arbitrarily small tubular neighborhood of $\gamma$, for some $\ell(p_i)\geq 1$. We denote the union of these segments respectively by $W^s(p_i)$ and $W^u(p_i)$. As in the previous case, there is a circular order and $f^{q_i}$ maps a curve in one of the families to the next curve in the same family, according to this order.  The map $f^{q_i\ell(p_i)}$ is the first iteration of $f$ that fixes $p_i$ and each curve of the two families. The Lefschetz index of $p_i$ is
$$
ind(f^{nq_i}, p_i)=\left\{\begin{matrix} 1 & \mbox{if}\,\,\, n\not\equiv 0 \mod{\ell(p_i)}\\  1-\ell(p_i)& \mbox{if}\,\, \, n \equiv 0 \mod{\ell(p_i)}\end{matrix}\right.
$$
We say that $p_i$ is a $\ell(p_i)$-pronged negative Mather sectorial periodic point of $f$ with $\ell(p_i)\geq 1$.

\begin{remark}
If such a point $p_i$ is 1-pronged, then $S$ has local multiplicity 1 along the corresponding negative hyperbolic periodic orbit $\gamma$. Consider the blow-up of $M$ along $\gamma$. The stable and unstable manifolds of $\gamma$ define two simple closed curves on the torus $\partial M_\gamma$, that turn once meridionally and twice longitudinally. In a meridian/longitude base of the homology of $\partial M_\gamma$, their class is $(1,2)$. The Birkhoff section defines also a simple closed curve in $\partial M_\gamma$ of class $(0,\pm1)$. Their (non-oriented) intersection number is $1\times |\pm 1| +2\times 0=1$. 

In order to avoid these points, we need a Birkhoff section whose homology class along $\partial M_\gamma$ is $(a,b)$ such that $|b|+2|a|>1$. Observe that being a Birkhoff section, $|b|\geq 1$. In what follows we build a Birkhoff section such that $|a|\geq 1$.
\end{remark}

We are left with the case in which a component $\gamma$ of $L$ is an elliptic periodic orbit of the flow. The corresponding points in $b_I\dot S$ are non-degenerate and the differential of $f$ is locally an irrational rotation. In particular their Lefschetz index is equal to 1.  The corresponding point $p_i$ satisfies Zehnder's condition on elliptic periodic points by hypothesis (G3) and  is hence contained in the closure of the homoclinic intersections of $f$.
 
 \medskip
 
For the rest of the proof we need to avoid having 1-pronged negative Mather sectorial periodic points. This makes use of the condition (G2).

\begin{proposition}\label{prop: 1pronged}
A Reeb vector field satisfying  (G1), (G2) and (G3), has a \mbox{$\partial$-strong} Birkhoff section $S'$ such none of the points in $b_I\dot S'$ is 1-pronged.
\end{proposition}

Using the techniques developed in \cite{CDHR}, the idea of the proof is to construct $S'$ from the Birkhoff section $S$ by adding binding components. 

\begin{proof}
Let $L'\subset L=\partial S$ be the link of negative hyperbolic periodic orbits that correspond, when collapsed to a point, to Mather sectorial periodic points in $b_I\dot{S}$ that are 1-pronged. If $L'$ is empty, set $S'=S$. If not, the equidistribution of periodic orbits,  hypothesis (G2), implies that we can apply Proposition~2.11 of \cite{CDHR}. Hence, there exists a link $K\subset M\setminus L$ and a class $y\in H^1(M\setminus K;\mathbb{R})$ such that $\langle y, \gamma\rangle>0$ for every $\gamma \in L$.

Let $x\in H^1(M\setminus L;\mathbb{R})$ be the homology class Poincar\'e dual to $S$. Then, as in the construction of a Birkhoff section in \cite{CDHR}, there exists $n\in \mathbb{N}$ such that $y_n=y+nx$ is Poincar\'e dual to a Birkhoff section $S_n$ whose boundary is contained in $K\cup L$. More precisely, the class of $y_n\in H^1(M\setminus(L\cup K);\mathbb{R})$ satisfies the hypothesis of Theorem~A.1 of \cite{CDHR}. Observe that if $m>n$, then the homology class $y_m=y+mx$ is also dual to a Birkhoff section. In what follows, we assume that $n\geq 2$.

We claim that one of the $\partial$-strong Birkhoff sections, either $S_n$ or $S_{n+1}$,  is such that no point in its ideal compactifications $b_I\dot S_n$ or respectively $b_I\dot S_{n+1}$  is 1-pronged.
By construction, $\partial S_n\subset L\cup K$ and we need to prove that there are no components that give rise to 1-pronged periodic points. We do this by cases, considering only the negative hyperbolic periodic orbits in $\partial S_n$. 

For a periodic orbit $\gamma$ in $L'\cap \partial S_n$, the condition $\langle y, \gamma\rangle>0$ implies that $S_n$ defines a closed simple curve on $\partial M_\gamma$ whose homology class is the sum of $n$-times the homology class induced by $S$ plus the homology class induced by $y$. The homology class induced by each connected component  $\partial S$ that is mapped to $\gamma$ is $(0,\pm 1)$ and thus taking into account all the connected components of $\partial S$ that lie in $\gamma$,  $S$ induces the homology class  $(0,b_0)$ where $b_0\in \mathbb{Z}$.  Analogously, the homology class induced by $y$ is of the form $(a_1,0)$ for an integer $a_1>0$. Then the  class induced by $S_n$ on $\partial M_\gamma$ is of the form $(a_1,nb_0)$.  Since $\gamma \in \partial S_n$, $b_0\neq 0$. The number of intersection points with the stable or unstable manifolds of the negative hyperbolic orbit $\gamma$ is equal to $n|b_0|+2a_1\geq 2a_1\geq 2$ for every $n\geq 1$. Hence the corresponding periodic points in $b_I\dot S_n$ are not 1-pronged.

For a negative hyperbolic periodic orbit $\gamma\in (L\setminus L')\cap \partial S_n$, the induced homology class in $\partial M_\gamma$ is $n(a_0,b_0)+(a_1,0)=(na_0+a_1,nb_0)$ with $|b_0|+2|a_0|\geq 2$, $b_0\neq 0$ and $a_1>0$. The number intersection points with the stable or unstable manifolds of $\gamma$ is then equal to $|nb_0|+2|na_0+a_1|\geq 2$ for every $n\geq 2$. Hence the corresponding periodic points in $b_I\dot S_n$ are not 1-pronged.

Finally for a negative hyperbolic periodic orbit $\gamma \in K$, we know that $\gamma$ is in the complement of $L$ and hence positively transverse to $S$. This implies that the homology class induced by $S$ on $\partial M_\gamma$ is of the form $(a_0,0)$ with $a_0>0$. On the other hand, the homology class induced by $y$ is of the form $(a_1,b_1)$ with $b_1\neq 0$. Hence, the number of intersection points with the stable or unstable manifold of $\gamma$ is equal to $|b_1|+2|na_0+a_1|$. This number can be equal to $1$ if $na_0+a_1=0$ and $b_1=\pm 1$, in this case we change $S_n$ for $S_{n+1}$  so that the number of intersection points becomes $|b_1|+2|(n+1)a_0+a_1|=1+2a_0>2$. We obtain the $\partial$-strong Birkhoff section $S'$ such that $b_I\dot S'$ has no 1-pronged points.
\end{proof}

\begin{remark}
Proposition~\ref{prop: 1pronged} is the only part in the proof of Theorem~\ref{thm: explicit} where the equidistribution of periodic orbits is essential. Among the hypothesis (G1), (G2) and (G3), the only one that is not known to be generic among geodesic flows is (G2). As explained in the introduction, (G2) is replaced by density of closed geodesics.

The rest of the proof in this section applies to geodesic flows, so in order to prove Theorem~\ref{thm: geodesic} we will give in Section~\ref{sec: geodesic} another argument to avoid 1-pronged periodic points in $b_I\dot S$.
\end{remark}

\medskip

In order to simplify the notation, we denote by $S$ a Birkhoff section such that $b_I\dot S$ has no 1-pronged points. At this point we have a homeomorphism $f$ of $c_I\dot S$ such that
\begin{itemize}
\item $f|_{\dot S}$ is strongly non-degenerate and preserves the area form $d\lambda$;
\item the elliptic periodic points of $f$ satisfy Zehnder's condition;
\item the points in $b_I\dot S$ are periodic and are either elliptic or Mather sectorial (with at least 2-prongs);
\item $f$ preserves a measure that is positive on open sets.
\end{itemize}

\subsection{Mather sectorial periodic points}

\medskip

Observe that hypothesis (G2) in Theorem~\ref{thm: explicit} implies that the Reeb vector field has infinitely many periodic orbits and, combined with hypothesis (G3), we get infinitely many hyperbolic periodic orbits. 

Let $per(f)$ be the set of periodic points of $f$, that has infinite cardinality, and $fix(f)$ the set of fixed points of $f$. We denote further by $per_h(f)$ the set of positive and negative Mather sectorial periodic points of $f$, including the hyperbolic non-degenerate periodic points in $\dot S$, and $fix_h(f)$ the set of Mather sectorial fixed points of $f$. For every $n\in \N$, the set $fix_h(f^n)$ decomposes into $fix_h^\pm(f^n)$ according to the sign of the Lefschetz index of each point (observe that this $\pm$ sign is not the sign of the corresponding hyperbolic periodic orbit of $R$). In particular we have that $\#per_h(f^n)>2g-2$, the bound in the hypothesis of Theorem~\ref{thm: LS} and that we use in the following section.

\begin{remark}\label{remark: infinite}
As mentioned before, property (G2) is not generic among geodesic flows. Nonetheless, such a vector field satisfying Zehnder's condition (G3) always has infinitely many hyperbolic Reeb orbits provided it has at least one elliptic periodic orbit. On the other hand, a geodesic flow such that all its periodic orbits are hyperbolic has to have infinitely many periodic orbits. 
\end{remark}

\subsection{Results from surface dynamics}

As in Section~6 of \cite{CDHR}, we want to establish an equivalence relation among Mather sectorial periodic points (including Mather sectorial periodic points in $b_I\dot S$). The aim is to say that two such points $x_1$ and $x_2$ are equivalent if each of the stable and unstable branches of $x_1$ has the same closure as any of the branches of $x_2$. In \cite{CDHR}, the homeomorphism considered had no elliptic periodic points and the equivalence relation was only studied for periodic point in $\dot S$. In this section, we start by extending the equivalence relation to our map $f$. We then explain two results that we use in the proof of Theorem~\ref{thm: explicit}. 

The following statement corresponds to \cite[Lemma~6.5]{CDHR} and its proof is extracted from the proof of~\cite[Theorem~8.3]{KLN}. Here Zehnder's condition is used.

\begin{lemma}
\label{lemma_thm_8.3_KLN}
Let $U \subset c_I\dot S$ be an open $f$-invariant connected set, and let $p \in b_IU$ be regular and periodic.
Then $Z(p) \subset c_I\dot S$ is an annular continuum with no periodic points of $f$. 
\end{lemma}

\begin{proof}
We extend the proof from \cite{CDHR} to consider elliptic periodic points. By \cite[Corollary~7.2]{KLN} or \cite[Theorem~6.2]{CDHR}, $Z(p)$ is either a cellular continuum with a unique fixed point $x$ and no other periodic points, or an annular continuum without periodic points. 

Assume, by contradiction, that we have the first option. 
There are two cases to treat:
\begin{itemize}
\item there is a periodic point in $Z(p)$ that is elliptic;
\item every periodic point in $Z(p)$ is Mather sectorial.
\end{itemize}
The second case was treated in \cite{CDHR}, here we treat only the first case. Assume that $Z(p)$ contains an elliptic periodic point $y$.
Since the set $Z(p)$ contains more that one point, we can choose a neighborhood $V$ of $y$ that does not contains $Z(p)$. By Zehnder's condition there is a hyperbolic periodic point $z$ and a disc $D\subset V$ containing $y$ that is bounded by finitely many arcs in the stable/unstable manifolds of $z$. Let $\partial D=\cup_{i}^n W_i$ where each $W_i$ is an arc in the stable or unstable manifold of $z$ and $W_i$ intersects $W_{i\pm1}$ for every $i$, with the indices taken modulo $n$.

By \cite[Corollary~8.3]{Mather} if a branch of the stable/unstable manifold of $z$ intersects $Z(p)$, the whole branch is contained in $Z(p)$. Since $Z(p)$ is connected and contains points outside $D$, we have that $\partial D\cap Z(p)\neq \emptyset$. Hence there exists $1\leq i_0\leq n$ such that $W_{i_0}\cap Z(p)\neq \emptyset$, assume without loss of generality that $i_0=1$. Then $W_1\subset Z(p)$ implying that $W_2\cap Z(p)\neq \emptyset$. We conclude, iterating this argument $n$-times that $\partial D\subset Z(p)$. But $U$ is connected and intersects both $D$ and $c_I\dot U\setminus D$ because $Z(p)$ is not entirely contained in $V$ and $D\subset V$. Thus $U$ intersects $\partial D\subset Z(p)$ contradicting the fact that $Z(p)\subset \partial U$. Thus there are no elliptic periodic points in $Z(p)$.
\end{proof}

The following corollaries can be deduced for Mather sectorial periodic points, as in \cite{CDHR}.

\begin{corollary}[Corollary~8.7 from~\cite{KLN} or Corollary~6.6 from \cite{CDHR}]
\label{cor_8.7_KLN}
If $U \subset c_I\dot S$ is an open connected $f$-periodic set with $\#b_IU<\infty$, then the boundary of $U$ is a finite disjoint union of aperiodic annular continua and periodic points.
\end{corollary}

\begin{corollary}[Corollary~8.9 from~\cite{KLN} or Corollary 6.7 from \cite{CDHR}]
\label{cor_8.9_KLN}
Let $x \in c_I\dot S$ be a Mather sectorial periodic point of $f$.
If $x$ belongs to some periodic continuum $K \subset c_I\dot S$, then the stable and unstable manifolds of $x$ are contained in $K$.
Moreover, all the stable and unstable branches of $x$ have the same closure in $c_I\dot S$.
\end{corollary}

With the above results in place, we can consider a relation on the set of Mather sectorial periodic points of $f$.
Similarly to~\cite[Section~3]{LS}, a periodic point $x_1$ is related to a periodic point $x_2$ if each of the stable/unstable branches of $x_1$ has the same closure as any of the stable/unstable branches of $x_2$.
This relation is clearly symmetric and transitive.
Reflexivity is non-trivial and follows from Corollary~\ref{cor_8.9_KLN}.
Hence, this is an equivalence relation.
The set of equivalence classes is denoted by $\mathcal{E}(f)$.

For a set $A\subset per_h(f)$ we denote by $K(A)$ the union of the closures of the stable/unstable branches of the points in $A$, observe that if $A=\kappa$ an equivalence class,  then $K(\kappa)$ is the closure of one stable or unstable branch of any point in $\kappa$. If the set $A$ contains finitely many points, there exists $n\in \N$ such that $f^n$ fixes all these points as well as their stable and unstable branches. Denote by $L(A)$ the sum of the Lefschetz indices of these fixed points of $f^n$ multiplied by $-1$, in other terms 
$$L(A)=\sum_{p\in A}(\ell(p)-1),$$
with $\ell(p)=2$ if $p$ is non-degenerate and $\ell(p)\geq 2$ for every point in $A$. Hence if all the points in $A$ are non-degenerate, $L(A)=\# A$ and for any finite collection of points $L(A)\geq \#A$. If $A\subset per_h(f)$ is an infinite set, we define $L(A)=\infty$. 

\medskip

For the rest of the proof of Theorem~\ref{thm: explicit} we adapt the arguments in Sections 3 to 6 of \cite{LS}. The results from \cite{LS} that use Lefschetz indices need to be adapted to our setting, while the ones that depend only on the fact that the periodic points are either non-degenerate or Mather sectorial, apply directly to our setting. 

Recall that $f:c_I\dot{S}\to c_I\dot{S}$ is a homeomorphism preserving a measure that is positive on open sets and that $f$ is  non-degenerate in $\dot S$. The points in $b_I\dot S$ are periodic and their types are as in Section~\ref{sec: firstreturn}, without the 1-pronged periodic points. 

Let $V$ be a periodic regular domain of $f$ contained in $c_I\dot S$ (Definition~\ref{defn: regulardomain}). We recall the following results from Section~3 of \cite{LS}, whose proofs apply verbatim to our setting. 
\begin{enumerate}
\item  If $V$ is a periodic regular domain of $f$, then every equivalence class $\kappa\in \mathcal{E}(f)$ is either contained in $V$ or disjoint from $\overline{V}$. In the first case, $V$ contains all the branches of the points in $\kappa$ (Proposition~3.1 from \cite{LS}).
\item If $x \in  c_I\dot S$ is a Mather sectorial periodic point such that $x\in K(\kappa)$ with $\kappa$ an equivalence class, then $x\in \kappa$ (Corollary~3.2 from~\cite{LS}, see also Lemma~6.8 in \cite{CDHR}).
\item If $V$ is a periodic regular domain of $f$ that contains $k$ different equivalence classes $\kappa_i$, $1\leq i\leq k$, there are pairwise disjoint periodic regular domains $V_i\subset V$ such that $\kappa_i\subset V_i$ (Corollary 3.3 from \cite{LS}).
\end{enumerate}

We can now define the genus of an equivalence class, in the definition we consider the homology classes of loops and denote them by $[\cdot]$ and we use $\wedge$ to denote the intersection form in $H_1(S;\mathbb{R})$.

\begin{definition}
If $V\subset S$ is an open set $g(V)$ is the largest $s\in \Z$ such that there exist simple loops $\{\gamma_i\}_{0\leq j<2s}$ in $V$ such that:
\begin{itemize}
\item $\gamma_{2j}$ and $\gamma_{2j+1}$ intersect in a unique point and their homology classes satisfy $[\gamma_{2j}]\wedge [\gamma_{2j+1}]=1$;
\item if $i\neq j$ are two indices that do not satisfy the previous condition, then $\gamma_i\cap \gamma_j=\emptyset$.
\end{itemize}

The genus of a class $\kappa\in \mathcal{E}(f)$, denoted by $g(\kappa)\in \{0,1, \ldots, g(S)\}$, is the smallest number such that $\kappa$ is contained in a periodic regular domain of genus $g(\kappa)$.
\end{definition}

 Thus if we have a collection of equivalence classes $\kappa_i$ respectively contained in pairwise disjoint periodic regular domains $V_i$ we have that
$$\sum_{i=1}^{p} g(\kappa_i)\leq \sum_{i=1}^{p} g(V_i)\leq g(V),$$
where $\cup_{i=1}^pV_i\subset V$.

The following proposition is taken from \cite{LS}:
\begin{proposition}[Proposition 3.4 of \cite{LS}]\label{prop: 3.4_LS}
Let $\kappa \in \mathcal{E}(f)$ be an equivalence class of Mather sectorial periodic points and $V$ a periodic regular domain of genus $g(\kappa)$ such that $\kappa\subset V$. The for every finite family $(\kappa_i)_{i\in I}$ of equivalence classes included in $V$ and distinct from $\kappa$, there exists a finite family $(D_j)_{j\in J}$ of pairwise disjoint periodic regular open discs of $c_IV$ such that:
\begin{itemize}
\item[-]$\kappa \cap \overline{D_j}=\emptyset$ for every $j\in J$;
\item[-]for every $i\in I$ there exists $j\in J$ such that $\kappa_i\subset D_j$;
\item[-]for every $j\in J$ there exists $i\in I$ such that $\kappa_i\subset D_j$;
\item[-]if $f_I^{n_i}$ fixes all the points in $\kappa_i\subset D_j$, then $f_I^{n_i}(D_j)=D_j$;
\item[-]if $f_I^{n_i}(D_j)=D_j$, $f_I^{n_i}$ extends to the prime end compactification of $D_j$ without periodic points in the added circle.
\end{itemize}
\end{proposition}

The following result is analogous to~\cite[Proposition 4.1]{LS}:

\begin{proposition}
\label{prop: 4.1_LS}
For every non-empty equivalence class $\kappa\in \mathcal{E}(f)$,  
$$L(\kappa)> max(0,2g(\kappa)-2).$$
\end{proposition}

\begin{proof}
Since there are 
infinitely many hyperbolic periodic points, if there is a unique periodic regular domain it has to be equal to $c_I\dot{S}$ and there is a unique equivalence class $\kappa$ that contains infinitely many points. By definition $L(\kappa)=\infty$.

Now, assume that there are periodic regular domains different from $c_I\dot{S}$. We need to prove that if $\kappa$ is an equivalence class with finitely many points included in a periodic regular domain $V\neq S$ of genus $g(\kappa)>1$, then $L(\kappa)>2g(\kappa)-2$ (observe that the cases $g(\kappa)=0 $ and $g(\kappa)=1$ are trivial). Since $\kappa$ is finite, modulo replacing $f$ by a power of it, we can assume that $f$ fixes:
\begin{itemize}
\item every point in $\kappa$ as well as their stable/unstable branches;
\item the open set $V$;
\item $f_I$ fixes every point of $b_IV$,
\end{itemize}
where $f_I$ is the extension of $f|_V$ to the end completion $c_IV$.
Then for every $n\geq 1$ the set $fix(f_I^n)$ is finite and contains $b_IV$. Moreover, Corollary~\ref{coro: 2.3LS} implies that the fixed points in $b_IV\setminus b_I\dot S$ have Lefschetz index equal to 1.

Using Proposition~\ref{prop: 3.4_LS}, we can assume that $f_I$ has finitely many Mather sectorial periodic points. Under this hypothesis the proof of~\cite[Proposition~1.4]{LS} implies:

\begin{lemma}
There exists $n_1\in \N$ such that $tr((f_I)_{*,1}^{n_1})\geq 2g(\kappa)$.  
\end{lemma}

Observe that the lemma implies that  $tr((f_I)_{*,1}^{n_1m})\geq 2g(\kappa)$ for any $m\geq 1$.

\begin{proof}
We copy the argument from \cite{LS} here.  Lefschetz' formula implies that 
$$
\#fix_e(f_I^n)+\#fix_h^+(f_I^n)-\#fix_h^-(f_I^n|_{V\setminus b_I\dot S})+\sum_{p_i\in A_n}(1-\ell(p_i))=2-tr((f_I)^n_{*,1}),$$
where $A_n\subset b_I\dot S\cap c_IV$ are the Mather sectorial periodic points whose stable/unstable branches are also fixed by $f_I^n$. Recall that the points in $b_I V$ are fixed by $f_I$. Observe that the quantity 
$$\sum_{p_i\in A_n}(1-\ell(p_i))$$
is negative and uniformly lower bounded since the number of points in $b_I \dot S$ is finite. Thus
\begin{equation*} 
\begin{split}
\#fix_h(f_I^n)&\geq\#fix_h^-(f_I^n|_{V\setminus b_I\dot S})\\
 & = tr((f_I)^n_{*,1})-2+ \#fix_e(f_I^n)+\#fix_h^+(f_I^n)+\sum_{p_i\in A_n}(1-\ell(p_i))\\
& \geq tr((f_I)^n_{*,1})-2+\sum_{p_i\in A_n}(1-\ell(p_i)).
\end{split}
\end{equation*}
Since $\#per_h(f_I)$ is finite, the trace of $(f_I)_{*,1}^n$ is uniformly upper bounded. 

Let $(\xi_j)_{1\leq j\leq 2g(\kappa)}$ be the complex eigenvalues of $(f_I)_{*,1}$, then
$$tr((f_I)_{*,i}^n)=\sum_{j=1}^{2g(\kappa)} \xi_j^n.$$
Considering only the argument of the $\xi_j$, there is $n$ arbitrarily big such that the argument of every $\xi_j^n$ belongs to $(-\pi/4,\pi/4)$. For such an $n$ we have that
$$tr((f_I)_{*,i}^n)\geq \frac{\sqrt{2}}{2}\sum_{j=1}^{2g(\kappa)}|\xi_j^n|,$$
where $|\cdot|$ denotes the norm. Consequently $\sup_{1\leq j\leq 2g(\kappa)}|\xi_j|\leq 1$, implying that every $\xi_j$ is of norm 1, since the determinant of $(f_I)_{*,1}$ is equal one. Changing $\pi/4$ by an arbitrarily small $\delta>0$, we have that for every $\epsilon>0$ there exists $n$ such that
$$tr((f_I)_{*,1}^n)\geq (1-\epsilon) \sum_{j=1}^{2g(\kappa)}|\xi_j^n|=(1-\epsilon) 2g(\kappa).$$
Since the trace is an integer and $\epsilon$ is arbitrary, we get that there exists $n_1\in \mathbb{N}$ such that $tr((f_I)_{*,1}^{n_1})\geq 2g(\kappa)$.
\end{proof}

We continue with the proof of Proposition~\ref{prop: 4.1_LS}. Observe that there are finitely many equivalence classes included in $V$ that contain a fixed point of $f^{n_1}$, because $fix(f^{n_1})$ is finite. 

Proposition~\ref{prop: 3.4_LS} implies that there is a collection of pairwise disjoint periodic regular discs $\{D_i\}_{i\in I}$ invariant by $f_I^{n_1}$, whose union contains all the Mather sectorial fixed points of $f^{n_1}$ in $V\setminus \kappa$. Let $E_1\subset fix(f^{n_1})$ be the set of Lefschetz index 1 points in $V\setminus (\cup_i D_i)$, observe that $E_1\cap \kappa=\emptyset$. Let $E_2$ be the set of ends of $V$ that do not belong to any $D_i$, for $i\in I$. Then Lefschetz formula implies that
\begin{eqnarray*}
tr((f_I)^{n_1}_{*,1})-2 & = & -\sum_{p\in fix(f_I^{n_1})}ind(f_I^{n_1},p)\\ 
& = & L(\kappa)-\#E_1-\#E_2 -\sum_{i\in I}\left(\sum_{p\in D_i\cap fix(f_I^{n_1})}ind(f_I^{n_1},p)\right)\\
&= & L(\kappa)-\#E_1-\#E_2-\# I.
\end{eqnarray*}
The last equality above follows from the fact that each $D_i$ can be compactified to a closed disc by adding a circle of prime ends, without adding fixed points in the boundary circle. Thus for each $i$ we have that 
$$\sum_{p\in D_i\cap fix(f_I^{n_1})}ind(f_I^{n_1},p)=1.$$ 
Using the inequality $tr((f_I)_{*,1}^{n_1})\geq 2g(\kappa)$ we have, 
$$L(\kappa)\geq 2g(\kappa)-2+\#E_2+\#I.$$
Since $V$ is not equal to $c_I\dot{S}$, it has at least one end, implying that $\#E_2+\#I\geq 1$. This gives the result. 
\end{proof}

Now assume that there is an equivalence class $\kappa$ that contains a Mather sectorial periodic point $p$ with a homoclinic orbit: an orbit contained in the intersection of a stable branch and an unstable branch of $p$. By hypothesis (G1) in Theorem~\ref{thm: explicit}, this intersection is transverse. Then, if $p$ is non-degenerate the $\lambda$-lemma implies that any stable branch of $p$ intersects any unstable branch of $p$. This holds also for the degenerate Mather sectorial points in $\kappa\cap b_I\dot S$, since the dynamics is induced by a non-degenerate flow and the $\lambda$-lemma for flows implies that the corresponding branches intersect. Moreover, since all the stable and unstable branches of $\kappa$ have the same closure, we deduce that each stable branch of a point in $\kappa$ intersects all the unstable branches of the points in $\kappa$. In particular, every point in $\kappa$ has a homoclinic orbit in each of the stable/unstable branches. We say that $\kappa$ is a homoclinic equivalence class. From classical results, we have that if an equivalence class $\kappa$ contains a heteroclinic cycle then it is a homoclinic equivalence class. In order to prove Theorem~\ref{thm: explicit}, we need to prove that every equivalence class is homoclinic and for this we will prove that every class has a heteroclinic cycle.

We now prove a result analogous to Proposition~5.1 from \cite{LS}, that we need to adapt to our degenerate situation. In \cite{CDHR}, this result was used only for the hyperbolic periodic points in $\dot S$.

\begin{proposition}\label{prop: homoclass}
Let $V$ be a periodic regular domain of genus $g'\leq g$, where $g$ is the genus of $c_I\dot{S}$. A set  $A\subset per_h(f)\cap V$ such that $L(A)> 2g'$ contains a point whose equivalence class is homoclinic.
\end{proposition}

\begin{proof}
Let $\{p_i\}_{i\in I}$ be a finite family of distinct Mather sectorial periodic points in $V$ and assume that $L(\{p_i\}_{i\in I})>2g'$. We want to construct $L(\{p_i\}_{i\in I})$ loops in $V$ to then study their homology classes and find a heteroclinic cycle. We start by  fixing some notation. 

For $p_i$ a Mather sectorial periodic point let $q_i$ be its minimal period. Number the stable/unstable branches by $W^*_k(p_i)$ for $*=u,s$ and $1\leq k\leq \ell(p_i)$ if $p_i$ is $\ell(p_i)$-pronged, in a cyclic order in such a way that, in a sufficiently small neighborhood of $p_i$, we have that $W^s_{k-1}(p_i)$ is in the region or quadrant delimeted by $W^u_{k-1}(p_i)$  and $W^u_{k}(p_i)$ for any $2\leq k\leq \ell(p_i)+1$. 

Modulo replacing $f$ by a power of itself, we suppose that all the points $\{p_i\}_{i\in I}$ are fixed by $f$. The following lemma is analogous to Lemma~5.2 from \cite{LS}.  Set $\sim$ to be the equivalence relation among Mather sectorial periodic points defined above.

\begin{lemma}\label{lemma: discs}
There exists a family of closed discs $\{D_i\}_{i\in I}$ contained in $V$ such that
\begin{enumerate}
\item $D_i\cap f^{\ell}(D_i)=\emptyset$ for all $1\leq \ell \leq 2\ell(p_i)-3$;
\item $D_i=D_j$ if $p_i\sim p_j$;
\item $D_i\cap D_j=f^{\ell}(D_i)\cap D_j=\emptyset$ if $p_i\not\sim p_j$ and for every $1\leq \ell <2\ell(p_i)-3$;
\item the stable and unstable branches of $p_i$ intersect $D_i$ and do not intersect $D_j$ if $p_i\not\sim p_j$.
\end{enumerate}
\end{lemma}

\begin{proof}
Let $I'\subset I$ be such that for every $i\in I$ there is a unique $j\in I'$ such that $p_i\sim p_{j}$. For every $j\in I'$ choose a point $p_j'\in W^u_1(p_j)$ and let $D_j\subset V$ be a closed neighborhood of $p_j'$. Extend the collection $\{D_j\}_{j\in I'}$ to a collection $\{D_i\}_{i\in I}$ that satisfies condition (2). If $p_i$ and $p_j$ are two non-equivalent points then $K(p_j)$, the closure of the stable and unstable branches of $p_j$, does not contains $p_i$. Hence $K(p_i)\cap K(p_j)=\emptyset$. This implies that we can take the discs $D_i$ small enough so that conditions (1), (3) and (4) are satisfied.
\end{proof}

Observe that for $p_i\sim p_j$, then numbers $\ell(p_i)$ and $\ell(p_j)$ do not need to be equal. The discs $\{D_i\}_{i\in I'}$ need to satisfy condition (1) and (3) for all the $\ell(p_j)$ among the $p_j$ equivalent to $p_i$.

Set $\ell(\kappa)=\max_{p\in\kappa\cap \{p_i\}_{i\in I}} \ell(p)$ and $\kappa_i$ to be the equivalence class of the point $p_i$.

For $i\in I$ we consider the point $p_i$. We now construct $\ell(p_i)-1$ loops denoted by $\lambda^u_{i,\ell}$ for  $1\leq \ell \leq \ell(p_i)-1$, in the following way:
\begin{itemize}
\item if $\ell(p_i)=2$, as in the non-degenerate case, we consider the first two points $x^u_{i,1}$ and $y^u_{i,2}$ where the unstable branches $W^u_1(p_i)$ and $W^u_2(p_i)$ intersect $D_i$ respectively, and let $\alpha^u_{i,1}$ be the simple path between these two points contained in $W^u(p_i)$ that passes through $p_i$. Observe that the choice of the points $x^u_{i,1}$ and $y^u_{i,2}$ implies that $\alpha^u_{i,1}\cap D_i=\{x^u_{i,1}, y^u_{i,2}\}$. Choose now a simple path $\beta^u_{i,1}$ between $x^u_{i,1}$ and $y^u_{i,2}$ contained in $D_i$. Set $\lambda^u_{i,1}$ to be the concatenation of these two paths.
\item if $\ell(p_i)>2$ then $p_i\in b_I\dot S$. Recall that $\ell(p_i)\leq \ell(\kappa_i)$. By Lemma~\ref{lemma: discs} (1), the discs $f^{\ell}(D_i)$  for $0\leq \ell \leq 2\ell(\kappa_i)-3$ are two by two disjoint.

For $1\leq \ell \leq \ell(p_i)-1$ set $x^u_{i,\ell}$ and $y^u_{i,\ell+1}$ to be  the first two points where the unstable branches $W^u_{\ell}(p_i)$ and $W^u_{\ell+1}(p_i)$ intersect the disc $f^{(\ell-1)}(D_i)$ respectively, and let $\alpha^u_{i,\ell}$ be the simple path between these two points contained in $W^u(p_i)$ that passes through $p_i$. 
To form a loop choose now a simple path $\beta^u_{i,\ell}$ between $x^u_{i,\ell}$ and $y^u_{i,\ell+1}$ contained in 
$f^{(\ell-1)}(D_i)$.

Set $\lambda^u_{i,\ell}$ to be the concatenation of $\alpha^u_{i,\ell}$ and $\beta^u_{i,\ell}$.
Since along $W^u_{\ell+1}(p_i)$ the point $y^u_{i,\ell+1}$ lies between $p_i$ and $x^u_{i,\ell+1}$, the loops $\lambda^u_{i,\ell}$ and $\lambda^u_{i, \ell+1}$ intersect (at least) along the segment of $\alpha^u_{i,\ell}$ between $p_i$ and $y^u_{i,\ell+1}$.
\end{itemize}

Analogously, we construct the loops $\lambda^s_{i,\ell}$ for $1\leq \ell\leq \ell(p_i)-1$ using the images of the discs $D_i$ by $f^{(\ell(\kappa_i)+\ell-2)}$. More precisely,  for $1\leq \ell \leq \ell(p_i)-1$ set $x^s_{i,\ell}$ and $y^s_{i,\ell+1}$ to be  the first two points where the stable branches $W^s_{\ell}(p_i)$ and $W^s_{\ell+1}(p_i)$ intersect the disc $f^{(\ell(\kappa_i)+\ell-2)}(D_i)$ and set $\alpha^s_{i,\ell}$ the simple path between these two points contained in $W^s(p_i)$ that passes through $p_i$. Choose now a simple path $\beta^s_{i,\ell}$ between $x^s_{i,\ell}$ and $y^s_{i,\ell+1}$ contained in $f^{(\ell(\kappa_i)+\ell-2)}(D_i)$. Set $\lambda^s_{i,\ell}$ to be the concatenation of these two paths.

Observe that in total we constructed $L(\{p_i\}_{i\in I})$ unstable loops  and $L(\{p_i\}_{i\in I})$ stable loops. We now study their possible intersections. Here $[\lambda]$ stands for the homology class of a loop $\lambda$ in $H_1(c_IV;\Z_2)$ and $\wedge$ is the intersection form.

\begin{lemma}[Lemma 5.3 from \cite{LS}]\label{lemma: loops}
We have the following assertions:
\begin{enumerate}
\item if $W^u(p_i)\cap W^s(p_i)=\{p_i\}$ then for every $1\leq m,\ell \leq \ell(p_i)-1$, the intersection numbers are  $[\lambda^u_{i, \ell}]\wedge [\lambda^s_{i,m}]=1$ if $m=\ell$ or $m=\ell-1$ and $[\lambda^u_{i, \ell}]\wedge [\lambda^s_{i,m}]=0$ in the rest of the cases;
\item if $i\neq j$ and $W^u(p_i)\cap W^s(p_j)=\emptyset$ then  $[\lambda^u_{i, \ell}]\wedge [\lambda^s_{j,m}]=0$ for  every $1\leq \ell \leq \ell(p_i)-1$ and every $1\leq m\leq \ell(p_j)-1$;
\item if $p_i\not\sim p_j$ then $[\lambda^u_{i, \ell}]\wedge [\lambda^s_{j,m}]=0$ for  every $1\leq \ell \leq \ell(p_i)-1$ and every $1\leq m\leq \ell(p_j)-1$.
\end{enumerate}
\end{lemma}

\begin{proof}
Observe first that the intersection between a  $\beta^u$-path and a $\beta^s$-path is always empty, since these paths are contained in disjoint discs: the $\beta^u$-paths are contained in the first $\ell(\kappa_i)-2$ iterations of $D_i$, while the $\beta^s$-paths are contained in the following iterations. 

We start by proving assertion (1).
Take $\ell$ and $m$ in $\{1,2,\ldots, \ell(p_i)-1\}$ and $i\in I$. We claim that $\alpha^u_{i,\ell}\cap \beta^s_{i,m}=\emptyset$. Indeed, $\alpha^u_{i,\ell}\cap f^{(\ell(\kappa_i)+m-2)}(D_i)$ is empty since the first intersection points of $W^u_\ell(p_i)$ and $W^u_{\ell+1}(p_i)$ with $ f^{(\ell(\kappa_i)+m-2)}(D_i)$ are not in $\alpha^u_{i,\ell}$ because points along the unstable branches get farther from $p_i$ when $f$ is applied, proving the claim. 

Now, since for every $k\geq 1$ we have that $f^k(\alpha_{i,m}^s)\cap f^{\ell(\kappa_i)-2+m}(D_i)=\emptyset$, in particular $f^{\ell(\kappa_i)-2+m-\ell+1}(\alpha_{i,m}^s)\cap f^{\ell(\kappa_i)-2+m}(D_i)=\emptyset$ for any $1\leq \ell\leq\ell(p_i)-1$ (because $\ell(\kappa_i)-2+m-\ell+1\geq 1$). Thus
$$\alpha_{i,m}^s\cap f^{\ell-1}(D_i)=\emptyset.$$ 
Since $\beta_{i,\ell}^u\subset f^{\ell-1}(D_i)$, we have that $\alpha_{i,m}^s\cap\beta_{i,\ell}^u=\emptyset$.
Then
$$\lambda^u_{i,\ell}\cap \lambda^s_{i,m}=\alpha^u_{i,\ell}\cap \alpha^s_{i,m}\subset W^u(p_i)\cap W^s(p_i)=\{p_i\}.$$

If $m=\ell$ or $m=\ell-1$, the two loops cross at $p_i$ and this intersection cannot be undone in homology, hence $[\lambda^u_{i, \ell}]\wedge [\lambda^s_{i,\ell}]=1$ for every $\ell$ and $[\lambda^u_{i, \ell}]\wedge [\lambda^s_{i,\ell-1}]=1$ for every $\ell>1$. Now, in the other cases, the two loops touch at $p_i$ and the intersection can be undone in an arbitrary small neighborhood of $p_i$, hence $[\lambda^u_{i, \ell}]\wedge [\lambda^s_{i,m}]=0$. This proves part (1).

\medskip

For $i\neq j$ in $I$, we have that 
$$ \alpha_{i,\ell}^u\cap \beta_{j,m}^s=\emptyset \qquad \beta_{i,\ell}^u\cap \alpha_{j,m}^s=\emptyset \qquad \beta_{i,\ell}^u\cap \beta_{j,m}^s=\emptyset.$$

For assertions (2) and (3) we have that 
$$\lambda^u_{i,\ell}\cap \lambda^s_{j,m}=\alpha^u_{i,\ell}\cap \alpha^s_{j,m}\subset W^u(p_i)\cap W^s(p_j)=\emptyset,$$
thus $[\lambda^u_{i, \ell}]\wedge [\lambda^s_{j,m}]=0$ for  every $1\leq \ell \leq 2\ell(p_i)-1$ and every $1\leq m\leq 2\ell(p_j)-1$. 
\end{proof}

Back to the proof of Proposition~\ref{prop: homoclass}, we have that the group $H_1(c_IV;\Z_2)$ has dimension $2g'$, where $g'$ is the genus of $V$, and that the set
$$\Lambda^s=\{[\lambda_{i,\ell}^s]\,|\, i\in I \, \mbox{and} \, 1\leq \ell \leq \ell(p_i)\}$$
has $L(\{p_i\}_{i\in I})>2g'$ loops.  Hence the set $\Lambda^s$ is linearly dependant and we can substract a linearly dependant subset $\Lambda^{s'}\subset \Lambda^s$ that is minimal with respect to this property:  every $[\lambda_{i,\ell}^s]\in \Lambda^{s'}$ is a linear combination of the loops in $\Lambda^{s'}\setminus [\lambda_{i,\ell}^s]$.

\begin{lemma}\label{lemma: intersections}
For every $p_i$ such that there is an $1\leq \ell\leq \ell(p_i)-1$ with $[\lambda^s_{i,\ell}]\in \Lambda^{s'}$ there exists $p_j$ such that an unstable branch of $p_i$ intersects a stable brach of $p_j$.
\end{lemma}

\begin{proof}
Suppose that for every $p_j\neq p_i$ such that there is an $m$, with  $1\leq m \leq \ell(p_j)-1$, such that  $[\lambda^s_{j,m}]\in \Lambda^{s'}$, we have that $W^u(p_i)\cap W^s(p_j)=\emptyset$. Then we have that $[\lambda^u_{i,\ell}]\wedge [\lambda^s_{j,m}]=0$ for every $1\leq \ell \leq \ell(p_i)-1$ and every $1\leq m\leq \ell(p_j)-1$. Take $n$ such that $[\lambda^s_{i,n}]\in \Lambda^{s'}$, then we know that $[\lambda^s_{i,n}]$ is a linear combination of the rest of the loops in $\Lambda^{s'}$. This implies that $[\lambda^u_{i,\ell}]\wedge [\lambda^s_{i,n}]=0$ for every $1\leq \ell \leq 2\ell(p_i)-1$ a contradiction to Lemma~\ref{lemma: loops} (1).
\end{proof}

Lemma~\ref{lemma: intersections} implies that among the points $p_i$ such that there is an $1\leq \ell \leq 2\ell(p_i)-1$ such that $[\lambda^s_{i,\ell}]\in \Lambda^{s'}$, there is a heteroclinic cycle. The periodic points $p_i$ in the heteroclinic cycle belong to the same equivalence class and have homoclinic orbits. Hence this class is homoclinic, proving Proposition~\ref{prop: homoclass}.
\end{proof}

\bigskip

We can now prove Theorem~\ref{thm: explicit}, the rest of the proof follows verbatim Section~6 from \cite{LS}. The result that is needed is the following.

\begin{proposition}
Every equivalence class is homoclinic.
\end{proposition}

\begin{proof}
Choose a non-empty equivalence class $\kappa$ and let $g$ be the genus of $c_I\dot{S}$. If $\kappa$ contains infinitely many periodic points, there is a finite subset $A\subset \kappa$ such that $L(A)>2g(\kappa)$. Proposition \ref{prop: homoclass} implies that $\kappa$ is homoclinic. Thus every equivalence class with infinite cardinality is homoclinic. 

We now assume that $\#\kappa<\infty$. By Proposition~\ref{prop: 4.1_LS}, $L(\kappa)\geq\max(1,2g(\kappa)-1)$. Let $V$ be a periodic regular domain of genus $g(\kappa)$ containing $\kappa$. 

If $g(\kappa)=0$, we have $L(\kappa)\geq 1>0=2g(\kappa)$ and  Proposition \ref{prop: homoclass} implies that $\kappa$ is homoclinic. We consider thus the case $g(\kappa)>0$ and the case that is not covered by Proposition \ref{prop: homoclass} is when $L(\kappa)$ is equal to  $2g(\kappa)$ or to $2g(\kappa)-1$. Assume that we have an equivalence class $\kappa$ such that $2g(\kappa)\geq L(\kappa)\geq 2g(\kappa)-1$. Observe that for any $m\geq 3$ we have that
\begin{equation}\label{eq:inequality}
m(2g(\kappa)-1)>2m(g(\kappa)-1)+2.
\end{equation}

Since $V$ has strictly positive genus, we can find two simple loops $\lambda_1$ and $\lambda_2$ in $V$ that intersect in a unique point and whose algebraic intersection number is 1. Let $S'$ be the $m$-cover space of $c_I\dot{S}$ obtained by cutting and gluing cyclically $m\geq 3$ copies of $c_I\dot{S}\setminus \lambda_1$. Set $\pi:S'\to c_I\dot{S}$ to be the projection. Observe that $S'$ is a closed surface of genus $m(g-1)+1$ and $V'=\pi^{-1}(V)$ is connected and has genus equal to $m(g(\kappa)-1)+1$. 

Let $z_0\in c_I\dot{S}$ be a fixed point of $f$ and fix a lift $z_0'\in S'$. Consider the fundamental groups $\pi_1(S',z_0')$  and $\pi_1(c_I\dot{S}, z_0)$, and denote by $f_*:\pi_1(c_I\dot{S},z_0)\to\pi_1(c_I\dot{S},z_0)$ be the application induced by $f$. The image $\pi_*(\pi_1(S',z_0'))$ is a subgroup of index $m$ of $\pi_1(c_I\dot{S},z_0)$ and the application $f_*$ acts as a permutation on the set of subgroups of index $m$ of $\pi_1(c_I\dot{S},z_0)$. Since $\pi_1(c_I\dot{S},z_0)$ has finitely many subgroups of index $m$ (we refer to \cite{Hall}), there exists $q\geq1$ such that $f_*^q$ fixes $\pi_*(\pi_1(S',z_0'))$ and by the lifting theorem $f^q$ induces a homeomorphism $f':S'\to S'$, satisfying the following conditions:
\begin{itemize}
\item the periodic points of $f'$ projet to periodic points of $f$ and their types coincide (elliptic or $\ell$-pronged Mather sectorial);
\item $f'$ preserves a measure that is positive on non-empty open sets (the lift of the measure preserved by $f$);
\item $V'$ is a periodic regular domain of $f'$;
\item $\pi^{-1}(\kappa)\subset per_h(f')$ and 
$L(\pi^{-1}(\kappa))=mL(\kappa)\geq m(2g(\kappa)-1)$. Equation \eqref{eq:inequality} implies hence that $L(\pi^{-1}(\kappa))>2m(g(\kappa)-1)+2=2g(V')$.
\end{itemize}
We can now apply Proposition~\ref{prop: homoclass} to conclude that $\pi^{-1}(\kappa)$ is a homoclinic class. Every homoclinic orbit in $V'$ projects to a homoclinic orbit in $V$ and hence $\kappa$ is a homoclinic class.
\end{proof}

\section{Geodesic flows: a new proof of Theorem \ref{thm: geodesic}}\label{sec: geodesic}

In this section we give a proof of the existence of homoclinic orbits for every hyperbolic periodic point of a $C^\infty$-generic geodesic flow. In order to use the proof of Theorem~\ref{thm: explicit} in the case of geodesic flows we have to address te following two things:
\begin{enumerate}
\item Find a Birkhoff section $S$ such that in the closed surface $c_I\dot S$, obtained by collapsing each connected component of $\partial S$ to a point, there are no 1-pronged points. In other words, we need to prove Proposition~\ref{prop: 1pronged} for generic geodesic flows.
\item Establish the genericity of Zehnder's condition (G3).
\end{enumerate}
Observe that condition (G1),  strongly non-degenerate, is known to be generic among geodesic flows as explained in Section~2 of \cite{CP}. Also, recall that Remark~\ref{remark: infinite} implies that the geodesic flows we consider have infinitely many hyperbolic periodic orbits.  Points (1) and (2), provide a proof of the existence of homoclinic intersections on any branch of a hyperbolic periodic orbit that does not depends on the equidistribution of periodic orbits (hypothesis (G2) in Theorem~\ref{thm: explicit}).

\medskip

Let us first discuss Zehnder's condition. As explained in the introduction, when we consider the Poincar\'e map of an elliptic periodic orbit we want to be able to find in any given neighborhood of the fixed point a topological disc containing the fixed point and bounded by finitely many pieces of stable/unstable manifolds of a hyperbolic periodic orbit. This is known to be generic among geodesic flows as we explain now. 

In the case of a strongly non-degenerate geodesic flow, elliptic periodic orbits are studied in Section 3 of \cite{C_Annals}. In particular, in the proof of Theorem~C, Contreras explains that among the metrics of class $C^4$ there is a $C^k$-generic set, for every $k\geq 4$, such that  in every neighborhood of an elliptic periodic orbit (that is a 1-elliptic periodic orbit in his notation because we consider only flows in three dimensions) the Poincar\'e map has  an invariant annulus, on which the map is a twist map. This type of maps are studied in the monograph by Le Calvez \cite{LC}, where he deduces the existence of hyperbolic periodic orbits contained in such an annulus and having homoclinic intersections satisfying Zehnder's condition (see the remark in page 34 of \cite{LC}). This settles point (2) above.

\medskip

We now find a suitable Birkhoff section.

\begin{proposition}\label{prop: 1prongedgeod}
Let $\Sigma$ be a closed oriented surface equipped with a Riemannian metric such that closed geodesics are dense in $\Sigma$. A geodesic flow satisfying (G1), has a $\partial$-strong Birkhoff section $S'$ such that none of the points in $b_I\dot S'$ is 1-pronged.
\end{proposition}

\begin{proof}
As in the proof of Proposition~\ref{prop: 1pronged} we start with any Birkhoff section $S$ of the geodesic flow of $\Sigma$. The existence of such a Birkhoff section follows from condition (G1), as proved in \cite{CM}. Assume that $c_I\dot S$ has 1-pronged points with respect to the dynamics of the extended first return map $f$.

Let $K=\pi(\partial S)$, where $\pi$ is the projection from the unitary tangent bundle $UT\Sigma$ to $\Sigma$. Then $K$ is a collection of closed geodesics in $\Sigma$. Among them, we consider $L\subset K$ the closed geodesics that are projections of the negative hyperbolic periodic orbits that correspond to the 1-pronged points in $c_I\dot S$. 

\begin{lemma}\label{lemma: Bannulus}
For every $\gamma\in L$ there exists a closed geodesic $\sigma_\gamma\not\in K$ such that $\sigma_\gamma$ intersects $\gamma$ transversely.
\end{lemma}

Before proving the lemma, we explain how this provides with the Birkhoff section $S'$ in the conclusion of Proposition~\ref{prop: 1prongedgeod}. Consider the set $\{\sigma_\gamma\,|\, \gamma\subset L\}$, 
and extract the minimal collection $L'$ such that for every $\gamma\in L$ there is a closed geodesic in $L'$ crossing $\gamma$ transversely. For each $\sigma\in L'$ we denote $\sigma^+$ and $\sigma^-$ the two periodic orbits of the geodesic flow that project to $\sigma$ (these correspond to the two possible orientations of $\sigma\subset \Sigma$). Consider one orientation of $\sigma$. Observe that if one considers the annulus $\mathbb{A}_{right}^\sigma \subset UT\Sigma$ of the unitary tangent vectors to the surface along $\sigma$ that point to the right of  the curve, this is a  section for the geodesic flow whose boundary is the union of $\sigma^+$ and $\sigma^-$, up to orientations. In the same way we can consider the annulus $\mathbb{A}_{left}^\sigma \subset UT\Sigma$. Thus from $\sigma$ we obtain two annuli that are sections for the flow, these are the Birkhoff annuli of $\sigma$.

Now take a connected component of $s\subset \partial S$ that projects to a curve in $\gamma\in L$ and consider $\sigma\in L'$ that intersects $\gamma$ transversely. There is at least one Birkhoff annulus of $\sigma$, lets call it simply $\mathbb{A}^\sigma$, whose interior intersects  $s$. We now use Fried's ideas \cite{friedanosov}, to add  $\mathbb{A}^\sigma$ to $S$ (we refer also to Section~3 of \cite{CDR}). Briefly, one can put the two surfaces in general position so that the intersections are transverse and then resolve the intersections preserving the transversality with the orbits of the flow. The obtained surface is a Birkhoff section $S_1$ whose boundary is  $\partial S\cup\{\sigma^\pm\}$. 

We claim that $s\subset \partial S_1$ does not corresponds to a 1-pronged point of $c_I\dot S_1$. The proof of this claim follows the proof of Proposition~\ref{prop: 1pronged}. In the manifold $M_s$ obtained by blowing up the periodic orbit $s$, we have that $S$ induces on the boundary torus  a first homology class that corresponds in a meridian/longitud basis to $(0,b_0)$ for some $b_0\in \mathbb{Z}$. The annulus $\mathbb{A}^\sigma$ induces the class $(a_1,0)$ with $a_1>0$, hence the addition of the two surfaces induces the class $(a_1,b_0)$. The intersection of the homology class $(a_1,b_0)$ with the curve in $\partial M_s$ that corresponds to the stable (or unstable) manifold of $s$ equals $|b_0|+2a_1\geq 2$. Hence the point that corresponds to $s$ in $c_I\dot S_1$ has at least two prongs.

Likewise, $\sigma^\pm\subset \partial S_1$ do not correspond to 1-pronged points of $c_I\dot S_1$. Take $\sigma^+$, without loss of generality, and assume it is a negative hyperbolic orbit (if not, the corresponding point in $c_I \dot S_1$ is not 1-pronged). Then $S$ induces in the boundary of $M_{\sigma^+}$ a homology class of the form $(a_0,0)$ for $a_0>0$, while the annulus $A^\sigma$ induces the class $(0,\pm 1)$. The fact that $A^\sigma$ does not turns in the meridional direction follows from the fact that the annulus does not intersects locally all the orbit passing near the periodic orbit $\sigma^+$. Then the number of intersections of $S_1$ with the stable/unstable manifold of $\sigma^+$ is equal to $1+2a_0>2$, implying that the corresponding point in $c_I\dot S_1$ is not 1-pronged.

Hence $S_1$ is a Birkhoff section of the geodesic flow that has, at least, one less 1-pronged point with respect to the original Birkhoff section $S$. Inductively we can treat each curve in $\partial S$ that projects to $L$, and obtain the Birkhoff section $S'$ as in Propostion~\ref{prop: 1prongedgeod}. 
\medskip

\noindent{\it Proof of Lemma~\ref{lemma: Bannulus}.} Consider a point $p\in \gamma$ such that $p$ does not belongs to any geodesic in $K \setminus \gamma$ and $U$ an open disc centered at $p$. Since $K$ is finite, the closure of the collection of closed geodesic of $\Sigma$ that do not belong to $K$ contains $U$. 

Recall that $\gamma$ is the projection of a negative hyperbolic periodic orbit, whose Morse index is odd. In particular, the Morse index of $\gamma$ is non-zero. This implies that a closed geodesic passing close to $\gamma$ has to intersect $\gamma$ and this intersection is transverse because both curves are geodesics of $\Sigma$. Then, for a sufficiently small neighborhood $U$, any closed geodesic $\sigma\not\in K$ intersecting $U$ intersects $\gamma$. This proves the lemma.
\hfill $\square$

\end{proof}

\section{Constrained and embedded Birkhoff sections}\label{sec: BSconstrained}

Here we give a proof of Theorem~\ref{thm: BS1}. As explained in the introduction we do this in three steps: first we construct a Birkhoff section containing $\Gamma$ in its boundary (Section~\ref{sec: constraint}); then we explain how to make it an embedded Birkhoff section while keeping $\Gamma$ in the boundary (Section~\ref{sec: embedded}) and finally we change again the Birkhoff section so that it contains $L'$ (Section~\ref{sec: legendrian}).

\subsection{A binding constraint}\label{sec: constraint}
In this section we prove the first part of Theorem~\ref{thm: BS1}: given a Reeb vector field admitting a $\partial$-strong Birkhoff section and  a collection $\Gamma$ of periodic orbits, we build a new $\partial$-strong Birkhoff section with $\Gamma$ contained in its boundary. Observe that hypothesis (G1) implies the existence of the $\partial$-strong Birkhoff section.

We start with a Reeb vector field $R$ of a contact closed 3-manifold $(M,\xi)$ satisfying hypothesis (G1), (G3) and such that every hyperbolic periodic orbit has a homoclinic in each of its stable/unstable branches. Observe that $R$ can be a geodesic flow.
As mentioned before, $R$ has a $\partial$-strong Birkhoff section $S$.
Let $\Gamma$ be a finite collection of periodic orbits of $R$. 
Arguing by induction, it is enough to treat the case where $\Gamma$ is a single periodic orbit $\gamma$ not in the boundary of~$S$, and to upgrade $S$ to a $\partial$-strong Birkhoff section that contains $\gamma \cup \partial S$ in its boundary.

By condition (G1), $\gamma$ is either hyperbolic or elliptic. We treat each case separately.

\smallskip

\noindent {\bf Case 1. $\gamma$ is hyperbolic.}
We apply a Fried-like construction to find a pair of pants~$P_\gamma$ that is transverse to~$R$ and contains~$\gamma$ in its boundary. Then we ``add it'' to~$S$ using the Fried sum operation as we now explain (see also~\cite{friedanosov}, \cite[Lemma 3.4]{CDR} or \cite[Appendix C]{CDHR}).

By hypothesis, $\gamma$ has a homoclinic orbit. 
Consider a small disc~$D$ transverse to~$\gamma$ and the first-return map~$f$ on~$D$ obtained by following the flow of~$R$. 
Call~$y_0$ the intersection of~$\gamma$ and~$D$. In $D$, the stable and unstable manifolds of $\gamma$ define two collections of curves such that any curve in the first collection is transverse to each curve of the second collection. In what follows, we say that the unstable curves are vertical and the stable curves are horizontal. We refer to a closed region contained in $D$ as a rectangle if its boundary is formed by four curves, two of them vertical and the other two horizontal. 

Choose a homoclinic orbit of $\gamma$ that intersects~$D$ along a sequence of points $(x_n)_{n\in\Z}$ such that $x_{n+1}=f(x_n)$, and when $n$ tends to~$\pm\infty$, $x_n$ tends to~$y_0$ along a stable and an unstable branches respectively. 
These two branches determine a quadrant~$Q\subset D$ around~$y_0$. 
For some large enough $N$, consider a rectangle~$S_N\subset Q$ that contains~$x_{-N}$ in its boundary, and such that the other vertical boundary is contained in a vertical curve that is  further, in the stable direction, than the vertical curve through~$x_N$. We take $S_N$ so that $y_0\notin S_N$.
Then $f^{2N}(S_N)\cap D$  contains another rectangle intersecting~$Q$ that contains~$x_N$ in its boundary and has a non-empty intersection with $S_N$. Up to extending~$S_N$, we can assume that $S_N\cap f^{2N}(S_N)$  does not contains any corner of $S_N$ or $f^{2N}(S_N)$. 
The expansion properties of the stable/unstable direction implies the existence of a fixed point of~$f^{2N}$ in~$S_N\cap f^{2N}(S_N)$. 
Call it~$y_1$ and call~$\gamma_1$ the corresponding periodic orbit of the flow of~$R$.

Now consider a rectangle $S'$ in~$Q$ bounded by the stable and the unstable branches of~$y_0$ that contains both~$y_0$ and $x_N$ in its boundary, and so that $f(S')$ contains both~$y_0$ and $x_{-N}$ in its boundary. 
Then the intersection~$S'\cap f^{2N}(S_N)$ contains a rectangle $S''$ in $D$ containing~$x_N$ in its boundary. 
The image of $S'$ under~$f^{2N+1}$ contains a tall and thin rectangle having $x_N$ in its boundary. Then $f^{2N+1}(S')$  intersects~$S''$.
Therefore $S''$ contains a fixed point of the map~$f^{2N+1}$. 
Call it $y_2$, and let~$y_3=f(y_2)\in f(S')\cap f^{2N+1}(S_N)$. 
Call~$\gamma_2$ the periodic orbit of the flow of~$R$ containing the points $y_2$ and $y_3$. 

Finally consider two curves $s_{02}$ connecting~$y_0$ and $y_2$ in~$D$, and $s_{13}$ connecting~$y_1$ and $y_3$.
By construction~$f(s_{02})$ is a curve in~$Q$ that connects $y_0$ and~$y_3$, while $f^{2N}(s_{13})$ connects $y_1$ and~$y_2$. 
Therefore the curves $s_{02}, f(s_{02}), s_{13}$ and $f^{2N}(s_{13})$ delimit in~$Q$ a quadrilateral~$P_{0132}$ with vertices~$y_0, y_1, y_3$, and $y_2$. 
Also, this implies that when pushing~$s_{02}$ under the flow~$R$ for a suitable time that might vary with respect to the point of $s_{02}$, we obtain a band~$P_{02}$ tangent to $R$  whose boundary is formed by~$s_{02}, \gamma, f(s_{02})$ and the arc of $\gamma_2$ between~$y_2$ and $y_3$. 
Similarly, when pushing~$s_{13}$ under the flow~$R$ for a suitable time, we obtain a band~$P_{13}$ tangent to $R$ whose boundary is formed by~$s_{13}, \gamma_1, f^{2N}(s_{13})$ and the arc of $\gamma_2$ between~$y_3$ and $y_2$. 
The union~$P_{0132}\cup P_{02}\cup P_{13}$ is then a 2-chain made of one part positively transverse to~$R$ and two parts tangent to~$R$, and whose boundary consists of~$\gamma\cup\gamma_1\cup \gamma_2$. 
It can then be smoothed along~$R$ into a section~$P_\gamma$ with the same boundary, and whose interior is everywhere positively transverse to~$R$. 

The union~$S\cup P_\gamma$ is then a singular surface transverse to~$R$. 
By resolving the singular intersection transversally to~$R$ by the Fried-sum process (as in~\cite[Lemma 3.4]{CDR}), we obtain a Birkhoff section $S_\gamma$ containing~$\gamma$ in its boundary. 

\smallskip
Note that depending on the quadrant~$Q$ defined by the considered homoclinic orbit, the multiplicities of the orbits~$\gamma\cup\gamma_1\cup \gamma_2$ in the boundary of~$S_\gamma$ vary: if $Q$ is between the stable and unstable branches in the clockwise order around~$\gamma$ then the oriented boundary of~$P_\gamma$ is~$-\gamma-\gamma_1+\gamma_2$; and if $Q$ is between the unstable and stable branches in the clockwise order around~$\gamma$ then $\partial P_\gamma$ is~$+\gamma+\gamma_1-\gamma_2$. 
Since by hypothesis there are homoclinics in all quadrants, we see that we can fix the mulitplicity of~$\gamma$ in~$\partial P_\gamma$ at $+1$ or~$-1$ at will. 
Since $\gamma\notin \partial S$, the multiplicity of $\gamma$ as boundary component of $S_\gamma$ is preserved by the Fried-sum. Hence, by repeating this operation, we can fix the multiplicity of $\gamma$ to be any integer value. 

\smallskip

\noindent {\bf Case 2. $\gamma$ is elliptic.}
By Zehnder's condition (G3), every tubular neighborhood of $\gamma$ contains a satellite hyperbolic orbit $\gamma'$. 
Write~$q$ for the the number of times $\gamma'$ wraps around~$\gamma$ in the longitudinal direction, or, said differently, the period of the intersection of~$\gamma'$ with a small transverse disc under the first-return map along the flow. 
Then there is an immersed annulus~$A$ in the considered neighborhood bounded by $q\gamma$ and $\gamma'$.
Moreover, since $\partial S$ has finitely many connected components, we can assume that $\gamma'$ is not in the boundary of the Birkhoff section $S$.

Fix $t>0$ so that $S$ intersects positively any arc of orbit of~$R$ of length~$\ge t$ that is not contained in $\partial S$. 
The annulus~$A$ is not necessarily transverse to~$R$, but there exist $m>0$ so that its intersection number with any arc of orbit of length~$t$ is in~$[-m,m]$. 
Therefore, for $n>m$, the intersection of the 2-chain $A+nS$ with any arc of orbit of length at least~$t$ is bounded below by $-m+n>0$. 
By Schwartzman-Fried-Sullivan Theory (see~\cite[Theorem A.1]{CDHR}) there is a Birkhoff section~$S'$ homologous to $[A]+n[S] \in H_2(M,\partial S \cup \gamma \cup \gamma';\Z)$. 
Since $\gamma$ is not in the boundary of~$S$ but is in the boundary of~$A$, it is also in the boundary of $S'$.

\subsection{Embedded Birkhoff sections}\label{sec: embedded}
In this section we prove the second step of Theorem \ref{thm: BS1}. 
We start from a $\partial$-strong Birkhoff section~$S$ whose boundary is only immersed and find a new $\partial$-strong Birkhoff section that is embedded. As in the previous section, this proof applies to geodesic flows.

By condition (G1), the boundary of~$S$ consists of elliptic and hyperbolic periodic orbits. 
Our goal is to reduce all of their multiplicities to~$\pm 1$ and for this we produce a new Birkhoff section with more boundary components. 
We first treat the elliptic orbits (all at the same time), to produce a new Birkhoff section embedded along the elliptic boundary orbits. The added boundaries are hyperbolic orbits. We then treat hyperbolic periodic orbits, one by one. 
The two steps follow (in reverse order) the two cases of the previous section, with some extra care in the elliptic case. 

Note that, assuming that a finite collection~$\Gamma$ is in~$\partial S$, one can ensure that under our operations the orbits of~$\Gamma$ remain with non-zero multiplicities, so that one can ensure that a fixed collection~$\Gamma$ is in the boundary of the constructed Birkhoff section.

\smallskip

{\bf Step 1. Treating the elliptic boundary orbits.}
Denote by $\eta_1, \dots, \eta_n$ the elliptic boundary orbits.
For each of them, choose a longitude.
Denote by $(p_1, q_1), \dots, (p_n, q_n)$ the boundary coordinates of~$S$ along these orbits in the chosen (meridian, longitude)-basis, with $p_i>0$.
Denote by~$\alpha_1, \dots, \alpha_n$ the respective irrational numbers such that in the same (meridian, longitude)-basis the differential of the flow has direction $(1, \alpha_i)$.
In particular the differential of the first-return map on a disc transverse to~$\eta_i$ is a rotation of angle~$2\pi\alpha_i$.
Note that changing a longitude changes the corresponding slopes $q_i/p_i$ and $\alpha_i$ by the same inverse of an integer.

There are four possible local configurations:
\begin{itemize}
\item if $q_i>0$ and $\alpha_i<q_i/p_i$, then the oriented boundary of~$S$ along $\eta_i$ is $+q_i\eta_i$,
\item if $q_i>0$ and $\alpha_i>q_i/p_i$, then the oriented boundary of~$S$ along $\eta_i$ is $-q_i\eta_i$,
\item if $q_i<0$ and $\alpha_i<q_i/p_i$, then the oriented boundary of~$S$ along $\eta_i$ is $-q_i\eta_i$,
\item if $q_i<0$ and $\alpha_i>q_i/p_i$, then the oriented boundary of~$S$ along $\eta_i$ is $+q_i\eta_i$.
\end{itemize}

In the first and fourth cases, we say that $\eta_i$ is a {\it positive} boundary orbit.
In the second and third cases we say it is {\it negative}.

Fix $\epsilon>0$ so that for every $i$ 
one has 
$$\epsilon<\frac{|p_i\alpha_i-q_i|}{|2q_i|+|p_i|}=\frac{|p_i\alpha_i-q_i|}{|2q_i|+p_i},$$
since $p_i>0$.

By Zehnder's generic condition~(G3), for every $i$, there exists an irrational number~$\beta_i\in(\alpha_i-\epsilon, \alpha_i+\epsilon)$ and an $R$-invariant neighborhood $U_i$ of~$\eta_i$ containing an $R$-invariant 2-torus whose flow has slope~$\beta_i$.
Moreover all orbits in~$U_i$ have direction between $(1, \alpha_i-\epsilon)$ and $(1, \alpha_i+\epsilon)$, and for every rational slope~$s/r$ between $\alpha_i$ and~$\beta_i$, there exists a hyperbolic periodic orbit in~$U_i$ with direction $(r, s)$.
In particular, any arc of orbit in~$U_i$ of length $t$ intersects the Birkhoff section $S$ a number in the interval
$[t(p_i(\alpha_i-\epsilon)-q_i), t(p_i(\alpha_i+\epsilon)-q_i)]$.

Therefore we can pick slopes $(r_1, s_1), \dots, (r_n,s_n)$ so that
\begin{itemize}
\item for every $i$, one has $\alpha_i-\epsilon<s_i/r_i<\alpha_i+\epsilon$,
\item for every $i$, $U_i$ contains a hyperbolic periodic orbit $\eta'_i$ of direction~$(r_i, s_i)$,
\item $s_1, \dots, s_n$ are distinct positive primes that are also coprime with all $q_1, \dots, q_n$.
\end{itemize}

By B\'ezout theorem, there exists $a, b_1, \dots, b_n\in\Z$ so that
\begin{itemize}
\item for every $i$ such that $\eta_i$ is positive, one has $aq_i+b_is_i=1$,
\item for every $i$ such that $\eta_i$ is negative, one has $aq_i+b_is_i=-1$,
\item $a>0$.
\end{itemize}

For every $i$ we consider a helicoidal annulus~$A_i$ bounded by $\eta'_i-s_i\eta_i$, just as an interpolation between $\eta'_i$ and $s_i\eta_i$.
Since its slope lies in the interval $[\alpha_i-\epsilon, \alpha_i+\epsilon]$, the intersection of an arc of orbit in~$U_i$ and $A_i$ of length $t$ lies in the interval~$[-2s_i\epsilon t, 2s_i\epsilon t]$.

We now consider the 2-chain $aS+\sum b_iA_i$ whose boundary along $\eta_i$ is $(aq_i+b_is_i)\eta_i=\pm \eta_i$, depending on the sign of~$\eta_i$ as boundary component of $S$.
Outside the open set~$\cup U_i$, the 2-chain is just $aS$, which cuts positively all orbits. 
What we need is to understand how $aS+\sum b_iA_i$ intersects orbits in~$\cup U_i$.

If $q_i$ is positive and $\alpha_i<q_i/p_i$ (first case in the above tetrachotomy), then $\eta_i$ is positive. 
Note that $b_i$ is negative in this case. 
An arc of length~$t$ in~$U_i$ cuts the considered 2-chain at least $a(t(p_i(\alpha_i-\epsilon)-q_i))+2b_is_i\epsilon t$ times.
Since $\eta_i$ is positive, one has $b_is_i=1-aq_i$, so 
\[
\begin{aligned}
a(p_i(\alpha_i-\epsilon)-q_i)+2b_is_i\epsilon & = ap_i\alpha_i-aq_i-ap_i\epsilon+2(1-aq_i)\epsilon \\
& > a(p_i\alpha_i-q_i)-\epsilon(2aq_i+ap_i).
\end{aligned}
\]
By the choice of~$\epsilon$, the latter is positive. 

The three other cases in the tetrachotomy can be handled by similar computations (up to changes of sign), and so the intersection is always positive. 

Therefore, by Schwartzman-Fried-Sullivan Theory (see~\cite[Theorem A.1]{CDHR}), the relative homology class of~$aS+\sum b_iA_i$ contains an Birkhoff section which has mutliplicity $\pm 1$ along~$\eta_i$, depending on the sign of~$\eta_i$, and $\mp 1$ along~$\eta_i'$.
We denote by~$S_1$ the obtained Birkhoff section. 
Its non-embedded boundary orbits are all hyperbolic.

\smallskip

{\bf Step 2. Treating the hyperbolic boundary orbits.}

Denote by~$\zeta_1, \dots, \zeta_m$ the orbits that constitute the boundary of~$S_1$ and whose multiplicity is different from~$\pm 1$.
Thanks to the previous paragraph, these are all hyperbolic periodic orbits.

This step is a bit easier than the previous one and we can treat all orbits one after another, so we just consider one hyperbolic orbit~$\zeta\in \partial S_1$, and explain how to change~$S_1$ so as to bring the multiplicity of~$\zeta$ to~$\pm 1$. 
Denote by~$m$ the multiplicity of~$\zeta$. 
Thanks to Theorem~\ref{thm: explicit}, $\zeta$ has homoclinic connections in all its quadrants. 

If $m>0$, we consider a quadrant $Q$ between the stable and unstable branches in the clockwise order around~$\zeta$, and a neighborhood of~$\zeta$ that contain no other component of~$\partial S_1$. 
We then apply the exact same Fried-like construction as in Case 1 of Section~\ref{sec: constraint} to produce two periodic orbits~$\zeta_1$ and $\zeta_2$ in~$Q$, and a pair of pants~$P_\zeta$ bounded by~$-\zeta-\zeta_1+\zeta_2$. 
The Fried-sum $S_1+P_\zeta$ is then a Birkhoff section, it has two more simple boundary orbits, and the multiplicity of~$\zeta$ has been decreased by~$1$. 

The case $m<0$ is treated in a similar way by considering $Q$ between the unstable and stable branches in the clockwise order around~$\zeta$. 

A direct induction concludes this step of the proof of Theorem~\ref{thm: BS1}.

\subsection{A Legendrian constraint}\label{sec: legendrian}
In this section, we prove the last part of Theorem~\ref{thm: BS1}. Observe that in Sections~\ref{sec: constraint} and \ref{sec: embedded}, hypothesis (G2) was not used. 

Let $(M,\xi)$ be a contact $3$-manifold. We take a Reeb vector field $R=R_\lambda$ for $\xi$, where $\lambda$ is a defining contact form 
satisfying the following $C^\infty$-generic conditions:
\begin{itemize}
\item $R$ is strongly non-degenerate (G1);
\item $R$ admits an embedded Birkhoff section $S$ (Section~\ref{sec: embedded});
\item hyperbolic periodic orbits of $R$ are dense in $M$, that is generic as a combination of conditions (G2) and (G3);
\item for every hyperbolic periodic orbit of $R$, every stable and unstable branches intersect contain a homoclinic orbit (Theorem \ref{thm: homoclinic}).
\end{itemize}

Let $L$ be a Legendrian knot in $(M,\xi)$. At the cost of a small Legendrian isotopy of $L$, we may assume that the intersections of $L$ and $S$ are transverse and contained in the interior of $S$.  
The strategy is to prove that one can always change the Birkhoff section and deform $L$, so that $L\cap S=\emptyset$ and $S$ intersects all the Reeb chords of $L$.
Once we have this, we can push forward $S$ with the flow of $R$ to obtain a Birkhoff section that contains $L$.

Along the proof, we will change of Birkhoff section by adding Fried pairs of pants constructed around a hyperbolic periodic orbit and move $L$ by $C^0$-small Legendrian isotopies. The Fried pairs of pants we add are embedded (see Proposition~\ref{lemma: embedded}) and adding them to $S$ does not changes the multiplicities of the periodic orbits in $\partial S$. At the end, we will have a Birkhoff section satisfying all the properties of Theorem~\ref{thm: BS1}. We proceed in two steps: in the first step we change the Birkhoff section  and the Legendrian knot so that the new ones are disjoint from each other. In the second step, we change again the Birkhoff section  and the Legendrian knot to guarantee that all the Reeb chords of the Legendrian intersect the section.

\medskip

\noindent {\bf Step 1. Disjoining $L$ from $S$.} 

\begin{lemma}\label{lemma: Lintersection}
There exists a $C^0$-small Legendrian isotopy of $L$ to a Legendrian knot $L_0$  and a Birkhoff section $S_0$ of $R$, such that $L_0 \cap S_0 =\emptyset$.
 \end{lemma}


\begin{proof}
Assume that $L\cap S\neq\emptyset$. As mentioned before, we can assume that $L\cap \partial S=\emptyset$ and the intersection between $L$ and $S$ is transverse. Consider a point $x\in L\cap S$ and an $\epsilon$-small neighborhood $U$ of $x$ in $S$ that contains no other intersection points. Since $L$ is transverse to $S$ at $x$, the contact plane $\xi_x$ is also transverse to $S$ and we may assume, by taking $U$ small enough, that the characteristic foliation $\xi U$ of $U$ is non-singular.

Since hyperbolic periodic orbits of $R$ are dense in $M$,  we can find a hyperbolic periodic orbit $h$ that intersects $S$ in $U$. Choose $y\in h\cap U$. Observe that by a $C^\infty$-small isotopy of $S$, we can change the embedding of the interior of $S$ so that it contains $y$ and such that the stable and unstable manifolds of $h$ at $y$ are transverse to the characteristic foliation at $y$.  We will say that a  Birkhoff section has property (*) at  a point $y$ of intersection with a hyperbolic periodic orbit when the condition above is satisfied.

Endow $M$ with a Riemannian metric so that $R$ is of norm one. Let $f:int(S)\to int(S)$ be the first return map and $y_-=f^{-1}(y)$. Let $T>0$ such that the image by the flow at time $T$ of $y_-$ is $y$, that is $\phi_R^T(y_-)=y$. Take $\epsilon>0$ and fix $d$ such that $\epsilon \ll d<T$. Take a point $p$ in the backward orbit of $y$ and a small disc $U_p$ transverse to the flow passing through $p$ so that any point in $U_p$ is at distance at least $d$ from $S\cup L$. Moreover, we can take $U_p$ so that the first intersection of the orbit of a point in $U_p$ with $S$ is contained in $U$. Let
$$V_p=\{\phi_R^{[0,t(q)]}(q)\,|\, q\in U_p,\, \phi_R^{[0,t(q)]}(q)\cap U=\phi_R^{t(q)}(q)\}.$$ 
 By Proposition~\ref{lemma: embedded}, we can find an $o(\epsilon)$-thin embedded Fried's pair of pants $P$ that is constructed from a  quadrilateral $Q_0\subset U_p$, and two tangent bands $R_1$ and $R_2$ that we call the ears. Observe that, by construction, the two ears $R_1, R_2$ intersect  $S$ in $U$. To obtain $P$ from $Q_0\cup R_1\cup R_2$, we smooth this surface and make its interior transverse to $R$.

The quadrilateral $Q_0$ is oriented so that it is positively transverse to $R$ and the ears $R_1, R_2$ are oriented so that $Q_0\cup R_1\cup R_1$ is an oriented surface in $M$ (we refer to Figure~\ref{fig: embpants2}). Consider the set $B_i\subset R_i$ defined as the component of $R_i\cap V_p$ for $i=1,2$ that intersects $Q_0$. Then the orientations of $B_1$ and $B_2$ face each other,  meaning that they intersect the characteristic foliation of $U$ with opposite signs. In other words, since we can take $Q_0$ as small as we want, the ears $R_1$ and $R_2$ are obtained by flowing two segments transverse to the flow that are almost parallel to the connected component of the intersection of the unstable manifold $W^u(h)$ with $U_p$ containing $p$. Then the curves $B_i\cap U$ are almost parallel to the connected component of $W^u(h)\cap U$ containing $y$. Using  property (*), the oriented curves $B_i\cap U$ are transverse to $\xi U$ and  intersect $\xi U$ in opposite directions.

\medskip

We now move the Legendrian knot. Note that for every $z$ in $U$, we can modify $L$ by a $C^\infty$-small Legendrian  isotopy (of order $\epsilon$) to a Legendrian knot $L'$ so that $L'$ intersects $U$ at $z$ and that the other intersections of $L$ and $S$ are unchanged. We call this property (**).

Recall that $P$ is the pair of pants transverse to $R$ obtained from $Q_0\cap R_1\cap R_2$ and let $S'$ be the Birkhoff section obtained as the Fried sum of $S$ and $P$: we consider $S\cup P$ and undo the intersections so that the interior of the obtained surface $S'$ is still transverse to $R$. Observe that this operation preserves the orientation of the surface, and is in particular performed along the curves $B_i\cap U$. 

 Using (**), we make a $C^\infty$-small Legendrian isotopy of $L$ so that $L'$ intersects $S$ along one of the curves $B_i\cap U$. Call the intersection point $z$. We choose $z$ so that after Fried sum, $L'\cap S'$ will  
not contain the point $z$, see Figure \ref{fig: resolution}. 

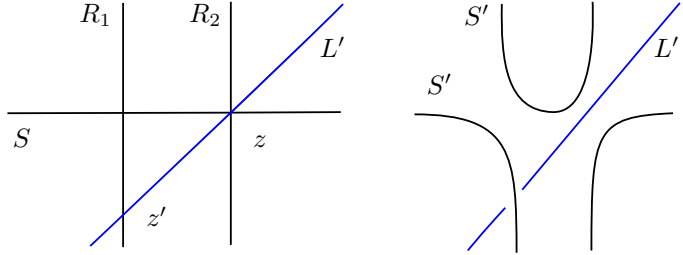
\begin{figure}[h]
\includegraphics[width=.7\textwidth]{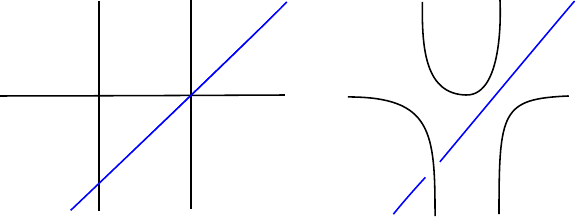}
\put(-160,40){$z$}
\put(-95,60){$S'$}
\put(-225,87){$R_1$}
\put(-81,86){$S'$}
\put(-135,74){$L'$}
\put(-184,87){$R_2$}
\put(-10,74){$L'$}
\put(-250,40){$S$}
\put(-200,10){$z'$}
\caption{The Fried sum of $P$ and $S$ that eliminates an intersection point in $L\cap S$. Here we have moved it to $z \in L'\cap R_2$.
The possible nearby intersection point $z'$ of $L'$ and $R_1$ can be eliminated by sliding $L'$ along the characteristic foliation of $R_1$, after first pushing $S$ up slightly along $R_1$ to avoid creating extra intersections with $S$ when sliding $L'$.}
\label{fig: resolution}
\end{figure}

We want to prove that, up to moving $L'$, we have that if $L\cap S$ has $n$ points, $L'\cap S'$ has $n-1$ points. First, we can assume again that $L'$ is transverse to $S'$. The intersection points of $L'$ with $S'$ are of three types: there are $n-1$ points in $L'\cap (S'\cap S)$, that is the part of $S'$ that coincides with $S$; the points that correspond to an intersection point in $L'\cap R_i$ for $i=1,2$ (except for the point $z$ that is no longer in $S'$); the points that correspond to an intersection point in $L'\cap Q_0$. Observe that by the choice of $Q_0$ there are no points in this last intersection.

Since the ears $R_1$ and $R_2$ are tangent to $R$  and arbitrarily thin, the characteristic foliations $\xi R_i$ are non-singular and transverse to the tangent boundary of $R_i$. Thus these foliations are by curves with endpoints in the two tangent boundaries of $R_i$. Take a point $q\in L'\cap R_i$. If the leaf of $\xi R_i$ passing through $q$ does not intersect $S$, this intersection point can be eliminated by a \emph{finger move} that consists in sliding $L'$ along the characteristic foliation of the ears, see Figure~\ref{fig: pant}. This operation is a $C^0$-small Legendrian isotopy. If the leaf of $\xi R_i$ passing through $q$ intersects $S$ at a point $q'$, the procedure above can create new intersections with $S'$. In this case, before sliding $L'$ we can push slightly  $S'$ near $q'$ in the (positive or negative) direction of $R$, away from the (short) leaf of $\xi R_i$ that passes through $q$. Thus we can undo all the intersections of $L'$ with $S'$ that correspond to intersections of $L'$ with $R_i$ for $i=1,2$.

At the end we obtain a Birkhoff section $S''$ that is basically $S'$ whose embedding in $M$ is changed by pushing parts of $S'$ along $R$, and a Legendrian knot $L''$ obtained from $L'$ by a $C^0$-small Legendrian isotopy that eliminates all the intersection points that correspond to intersection points in $L'\cap (R_1\cap R_2)$. Then $L''\cap S''$ is a collection of $n-1$ points, and by induction we can conclude that there is a Birkhoff section $S_0$ of $R$ and a Legendrian knot $L_0$ such that $L_0\cap S_0=\emptyset$ and such that $L_0$ is obtained from $L$ by a $C^0$-small Legendrian isotopy.
\end{proof}

\bigskip

Using Lemma~\ref{lemma: Lintersection}, we re-establish the notation: in $M$ endowed with the Reeb vector field $R$ as above, we have a Legendrian knot $L$ and a Birkhoff section $S$ such that $L\cap S=\emptyset$.

Let $S'$
be a small retract  of $S$ into $int(S)$. Take a neighborhood of $S'$ in $M$ containing $L$, that  is diffeomorphic to $S'\times [0,1]_t$, where $R$ is parallel to $\frac{\partial}{\partial t}$. Set $\pi : S' \times [0,1] \to S'\times \{0\}$ the projection along the Reeb direction. In what follows we write $S'$ for $S'\times \{0\}$. 
The map $\pi \vert_{L} : L \to S'$ is an immersion. The obstruction that prevents us from sliding $S$ along $R$ so that it contains $L$ are the double points of $\pi\vert_{L}$. These double points correspond to Reeb chords of $L$ contained in $S' \times [0,1]$. Up to deforming $L$, we can assume that these double points are transverse and thus in finite number and that there are no multiplicity three or more  points.

The idea of the proof below is to construct a Birkhoff section that cuts open each of the Reeb chords. As in Step 1, the new Birkhoff section is obtained from $S$ by adding Fried's pairs of pants.

\smallskip

\noindent {\bf Step 2. Elimination of double points.} 

We eliminate double points one by one.
Consider a double point $p$ of $\pi (L)$ in $S'$ and a  disk $D \subset S'$ around $p$ that meets $\pi (L)$ along two simple segments of curve.
Moreover if the disk $D$ is small enough, for every point $p' \in D$ there exists a $C^\infty$-small Legendrian isotopy of $L$ to a Legendrian $L'$ such that the double point of $\pi(L')$ is at $p'$. This property is denoted (***).
Recall here that small Legendrian isotopies of $L$ are in one-to-one correspondence with small isotopies of $\pi (L)$ in $S'$ with zero $d\lambda\vert_S$-flux, i.e. the $d\lambda$-area between $\pi (L)$ and its deformation is zero.
We take such a double point $p$ and consider the segment of Reeb orbit $p\times[0,1]$. Since $p$ is a double point, $p\times [0,1]$ intersects $L$ in two points at different $t$-altitudes, we assume that one intersection is  contained in $[0,1/4]$ and the other one in $[3/4,1]$. Let $c(p)$ be the Reeb chord contained in $S'\times(0,1)$ between these two points ($\pi(c(p))=p$).

By hypothesis, there exists a hyperbolic periodic orbit $h$ for $R$ in $M$ that intersects~$D$, and $h$ has a transverse homoclinic. Then Proposition~\ref{lemma: embedded} implies the existence of a close-by embedded Fried pair of pants $P$ that we construct by smoothing the surface constructed as the union of:
\begin{itemize}
\item a quadrilateral $Q_0$ transverse to $R$ and contained in $S' \times [1/3,2/3]$ such that $\pi (Q_0) \subset D$;
\item the ears $R_1$ and $R_2$ tangent to $R$, that have two edges attached to two adjacent edges of $Q_0$, the two edges tangent to $R$.
As before, these two ears can be taken very thin and thus we can assume that the characteristic foliations of $\xi R_1$ and $\xi R_2$ are made of non-singular arcs with endpoints in the tangent edges.
\end{itemize}

Now, thanks to the previous remark (***), we can isotop $L$ to $L_1$ by a $C^\infty$-small Legendrian isotopy so that: 
\begin{itemize}
    \item the Reeb chord $c(p)$ intersects $Q_0$ in its interior;
    \item $\pi(R_1\cup R_2)$ is  disjoint from the double points of $\pi|L$.
\end{itemize}
After this step, $L_1$ might intersect the ears $R_1$ and $R_2$, either along the components of $R_i \cap (S' \times [0,1])$, $i=1,2$, that meet $Q_0$, or along other components that cross from 
$S' \times \{0\}$ to $S' \times \{1\}$. We recall from Fried's construction,  see Proposition~\ref{lemma: embedded}, that the $R_i$, $i=1,2$, are as thin as we want and $L_1$ lies at a given distance from $S' \times \{0,1\}$. Therefore if $x\in R_i \cap L_1$ we can assume that the arc of the characteristic foliation $\xi R_i$ that passes through $x$ is contained in $S'\times (0,1)$. 

We want to move $L_1$ with a $C^0$-small Legendrian isotopy to a Legendrian knot $L_2$ so that $L_2\cap P=\emptyset$ and for this we prove that we can move $L_1$ to $L_2$ so that $$L_2\cap (Q_0\cup R_1\cup R_2)=\emptyset.$$
Given a component $B$ of $R_i \cap (S' \times [0,1])$, $i=1,2$, that crosses from $S' \times \{0\}$ to $S'\times \{1\}$, we can push all the intersections of $L_1 \cap B$ away from $B$ by a finger move as in Step 1, but we need to take extra care so that this move does not creates new double points in the projection of the Legendrian. Observe that $\pi (B)\subset S'$ is an arc and, if $N(B)$ denotes a small tubular neighborhood of $B$, the projection of $L_1 \cap N(B)$ is made of finitely many disjoint arcs transverse to $\pi (B)$  and we can isotop these arcs in $S'$ so that they are still disjoint and disjoint from $\pi (B)$. Moreover, this isotopy can be done using an arbitrarily small $d\lambda$-area in $S'$ that can be compensated elsewhere without affecting the number of double points.

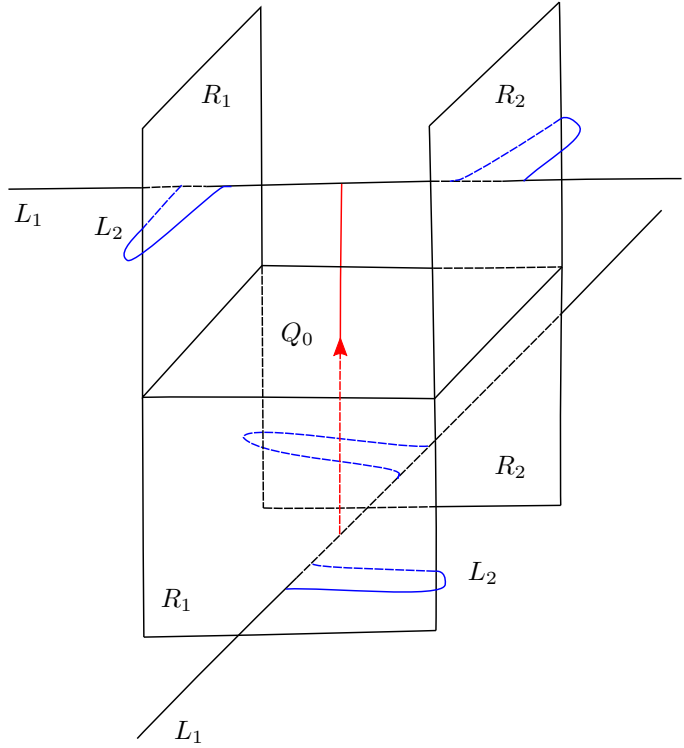
\begin{figure}[h]
\includegraphics[width=.7\textwidth]{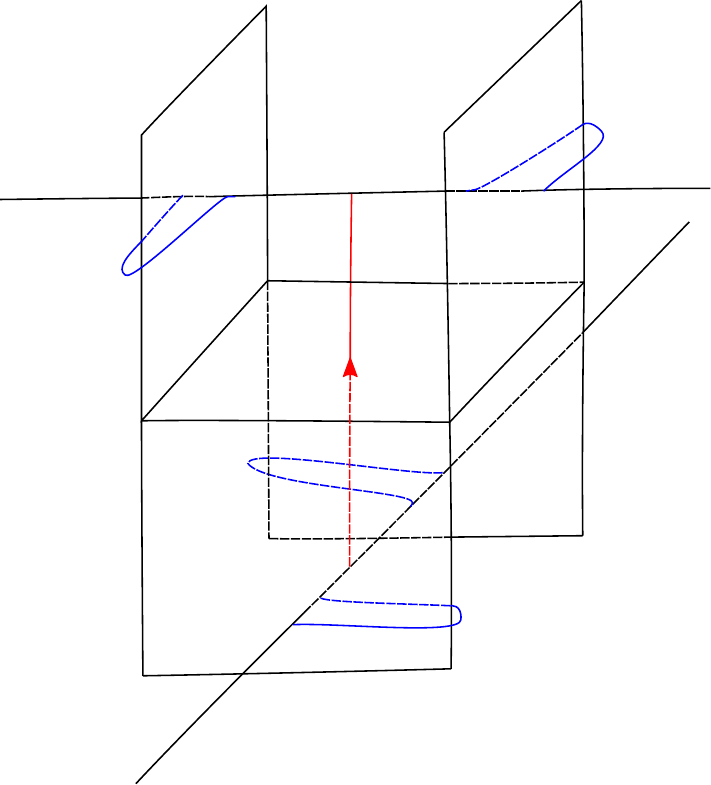}
\put(-150,150){$Q_0$}
\put(-220,190){$L_2$}
\put(-250,195){$L_1$}
\put(-190,0){$L_1$}
\put(-80,60){$L_2$}
\put(-180,240){$R_1$}
\put(-195,50){$R_1$}
\put(-70,240){$R_2$}
\put(-70,100){$R_2$}
\caption{The pair of pants $P$ (before smoothing) near $Q_0$. The Reeb vector field is vertical and the contact structure approximately horizontal. We push $L_1$ away from $P$ along its characteristic foliation $\xi P$ in counterclocwise manner to obtain $L_2$ (in blue), without creating new chords.
The Reeb chord intercepted by $Q_0$ is in red.}
\label{fig: pant}
\end{figure}

We are left with the four connected components $B_1, B_2, B_3$ and $B_4$ of 
$$(R_1\cup R_2)\cap (S'\times [0,1]),$$
that intersect $Q_0$. Consider their projection 
$$\pi (\cup_{i=1}^4 B_i) =\pi (\partial Q_0)\subset D.$$ The intersection of $L_1$ with a tubular neighborhood $N(\cup_{i=1}^4 B_i)$ of $\cup_{i=1}^4 B_i$ projects to disjoint small arcs transverse to the boundary of
$\pi (Q_0)$. We push each arc along the corresponding edge of $\pi (Q_0)$ in the direction of $\pi (\partial Q_0)$ oriented as the boundary of $\pi (Q_0)$. This insures that we do not create double points. See Figure \ref{fig: pant}.

After this operation, we have obtained a Legendrian knot $L_2$ that is obtained from  $L_1$ by a $C^0$-small Legendrian isotopy. Thus $L_2$ is Legendrian isotopic  to $L$ and is contained in $M\setminus (S\cup P)$, where $P$ is $Q_0\cup R_1\cup R_2$ smoothed and transverse to $R$ in its interior.
Fried summing $S$ and $P$, we obtain a new Birkhoff section $S_1$ for $R$  that is disjoint from $L_2$. Consider $S_1'$
 a small retract  of $S_1$ into $int(S_1)$ and a neighborhood of $S_1'$ in $M$ containing $L_2$, that  is diffeomorphic to $S_1'\times [0,1]_t$, where $R$ is the $\frac{\partial}{\partial t}$ direction. Set $\pi_1 : S_1' \times [0,1] \to S_1'\times \{0\}$  the projection along the Reeb vector field. Then $\pi_1|_{L_2}$ has (at least) one double point less than $\pi|_L$.  Observe further that the Birkhoff section $S_1$ is embedded: in one hand $S$ and $P$ are  embedded along their boundary and hence so is $S_1$; and the interior of $S_1$ is embedded.

Iterating this process, we end up with a Legendrian knot $L'$ that is Legendrian isotopic to $L$ and a Birkhoff section $S'$ disjoint from $L'$ such that $L'$ has no chord in $M\setminus S'$.
Thus we can slide $S'$ along $R$ so that the resulting surface contains $L'$. This finishes the proof of Theorem \ref{thm: BS1}.
\medskip

The properties needed to prove the Legendrian part of Theorem \ref{thm: BS1}: density of hyperbolic orbits, existence of homoclinics and of thin pairs of pants, are satisfied by transitive Anosov flows. Moreover, every smooth link can be made transverse to a flow by a small isotopy. The same proof applied in this smooth (non-contact, non-Legendrian) simpler context thus gives:

\begin{theorem} Let $X$ be a transitive Anosov flow on a closed $3$-manifold $M$ and $L\subset M$ be a smooth link. There exists a global surface of section $S$ for $X$ that contains  in its interior a link $L'$ isotopic to $L$ by a $C^0$-small isotopy.
\end{theorem}

\section{Existence of infinitely many Reeb chords}\label{sec: chords}

This section is devoted to the improvement of the Arnold's chord conjecture proved by Hutchings and Taubes, passing from one chord to infinitely many.

We first prove Corollary \ref{cor: chords}: for a $C^\infty$-generic Reeb vector field, up to a $C^\infty$-small Legendrian isotopy, a Legendrian knot has an exponentially growing number of honest Reeb chords with respect to the action. Recall that a Reeb chord is honest if it is not contained in a periodic orbit.

\begin{proof}[Proof of Corollary \ref{cor: chords}] 
We assume the $C^\infty$-generic property for the Reeb vector field $R$: hyperbolic periodic orbits are dense and have transverse homoclinic orbits. The fact that this is $C^\infty$-generic among the Reeb vector fields of a given contact structure follows from Theorem~\ref{thm: explicit}. Hence, given a Legendrian knot $L$, we can find a hyperbolic periodic orbit $h$ of $R$ that passes close to it, and thus a homoclinic orbit $\gamma$ for $h$ that passes close to $L$. We perform a $C^\infty$-small Legendrian isotopy of $L$ to a Legendrian knot $L'$ that intersects $\gamma$  transversely at a point $x$. Since the stable and unstable manifolds of $h$ are tangent to $R$, $L'$ is transverse to these submanifolds at $x$. 

Let $D$ be a disk  transverse to $h$. We now push $L'$ with the flow of $R$ in positive time, and consider the subsequent intersections with $D$. Let $\mathcal{O}^+(L')$ the union of the positive orbits of the points of $L'$, we want to consider the set $\mathcal{O}^+(L')\cap D$.  Observe that since $x\in \gamma\cap L'$, the positive orbit of $x$ accumulates along the stable manifold $W^s(h)\cap D$ on $h\cap D$. Each point in the intersection of the orbit of $x$ with $D$ is contained in an arc in $\mathcal{O}^+(L')\cap D$. Thus we obtain a sequence of arcs that accumulate on the connected component of $W^u(h)\cap D$ that contains the point $h\cap D$. If we now push $L'$ in negative time, we obtain analogously that the set $\mathcal{O}^-(L')\cap D$ contains a sequence of arcs in $D$ accumulating on the connected component of $W^s(h)\cap D$ that contains the point $h\cap D$. By the hyperbolicity of the orbit $h$, for sufficiently large times, these collections of arcs intersect transversally in $D$.
Each intersection point corresponds to an honest Reeb chord of $L'$. By standard hyperbolicity arguments, the number of intersection points  is exponentially growing with respect to time.
\end{proof}

We now focus on the proof of Theorem \ref{thm: chords}: given a closed 3-manifold $M$ and a contact form $\lambda$ whose Reeb vector field admits a Birkhoff section, every Legendrian knot has infinitely many Reeb chords,  and whenever there are finitely many disjoint ones, $M$ is a lens space or the sphere with a Reeb flow having exactly two periodic orbits.

\begin{proof}[Proof of Theorem~\ref{thm: chords}]
Let $S$ be a Birkhoff section for $R=R_\lambda$ and $L$ a Legendrian knot. 
We do not assume for the moment that $S$ is $\partial$-strong. 
We want to prove that $L$ has infinitely many chords.
Observe that if $L$ intersects a component $c$ of~$\partial S$, then it has infinitely many (dishonest) chords contained in the periodic orbit~$c$.

Now assume that $L\cap \partial S=\emptyset$. 
Let $h:int(S) \to int(S)$ be the first return map of the flow.
The set $\dot M = M \setminus \partial S$ is diffeomorphic to $int(S) \times \R / \sim$ where $(z,t) \sim (h(z),t-1)$.
The coordinate $t$ of $[0,1]$ can be arranged so that $R$ is everywhere positively proportional to $\frac{\partial}{\partial t}$.

We consider the image~$C$ of $L$ under the flow of $R$.
Then $C$ is an immersed cylinder in~$\dot M$ transverse to $int(S) \simeq int(S) \times \{0\}$. 

Assume that $L$ has finitely many chords.

\begin{lemma}
\label{lemma_injective}
There exists an immersed arc $\gamma : [0,1] \to S\cap C$ such that $\gamma\vert_{[0,1)}$ is injective and $\gamma(0)=\gamma(1)$.
\end{lemma}

\begin{proof}
As explained above, the open set $\dot M = M \setminus \partial S$ admits an infinite cyclic cover $int(S)\times \R$ and let $p: int(S) \times \R \to \dot M$ whose group of deck transformations is generated by the diffeomorphism $\phi:(z,t) \mapsto (h^{-1}(z),t+1)$.
The Reeb vector field lifts to a positive multiple of $\frac{\partial}{\partial t}$.
Since $L\cap \partial S=\emptyset$, we have that $L \subset \dot M$.
Orient $L$ so that the intersection number $k$ between $L$ and $S$ satisfies $k\geq0$. 

Assume first that $k=0$.
Let $\tilde L$ be any lift of $L$ to $int(S) \times \R$ by the map $p$.
Then~$p|_{\tilde L}$ defines a diffeomorphism $\tilde L \to L$.
In particular, $\tilde L$ projects to a closed loop in $int(S)$ via the projection $int(S) \times \R \to int(S)$.
Since this projection is contained in $C$, the desired arc can be obtained and the lemma is proved in this case.

From now on assume $k\geq 1$.
We can parametrize an oriented lift $L_0$ of $L$ by the map $p$ as an injective map $s \in \R \mapsto (c_0(s),t_0(s)) \in int(S)\times \R$ where 
\begin{equation}
\label{t_0_c_0}
\forall s\in\R: \qquad t_0(s+k) = t_0(s) + k, \qquad c_0(s+k) = h^{-k} \circ c_0(s).
\end{equation}
The map $c_0:\R\to int(S)$ is an immersion since $R$ is parallel to $\frac{\partial}{\partial t}$ and $L$ is Legendrian.
We consider the $k$ lifts $L_0,L_1,\dots,L_{k-1}$ of $L$ characterized by $L_j = \phi^j(L_0)$.
Moreover, $L_j$ can be parametrized by $s \in \R \mapsto (c_j(s),t_j(s))$ where $t_j,c_j$ are defined by $$ t_j(s) = t_0(s+j)-j, \qquad c_j(s) = h^{-j} \circ c_0(s+j). $$
In particular, $t_j$ and $c_j$ satisfy~\eqref{t_0_c_0}: $$ \forall s\in \R: \qquad t_j(s+k)=t_j(s)+k, \qquad c_j(s+k) = h^{-k} \circ c_j(s). $$
The set $C$ can be described as
\begin{equation}
C = c_0(\R) \cup c_1(\R) \cup \dots \cup c_{k-1}(\R) = \bigcup_{i\in\Z} h^{-i}(c_0([0,k])).
\end{equation}
If $c_0$ is not injective then the conclusion of the lemma follows.
Assume that there exists a periodic point $q \in c_0(\R)$ of $h$.
Let $m \geq 1$ be its primitive period and let $s_*$ satisfy $c_0(s_*)=q$.
Then
\begin{equation}
\begin{aligned}
c_0(s_*+km) &= h^{-k}(c_0(s_*+(m-1)k)) \\ &= \dots = h^{-km}(c_0(s_*)) = h^{-km}(q) = q = c_0(s_*)
\end{aligned}    
\end{equation}
and $c_0$ is not injective in this case.
We proceed assuming that $c_0$ is injective, so that all $c_j$ are injective as well, and the set $C = c_0(\R) \cup \dots \cup c_{k-1}(\R)$ contains no periodic points of $h$.
By the results of Hutchings and Taubes~\cite{HT1,HT2} there exists a Reeb chord for $L$, i.e. a smooth map $\eta:[0,T] \to M$ such that $\dot\eta = R\circ \eta$ and $\{\eta(0),\eta(T)\} \subset L$ for some $T>0$.
Choose a lift $\tilde\eta$ of $\eta$ by the map~$p$ such that $\tilde\eta(0) = (z_*,t_*) \in L_0$.
 Then $\tilde\eta(t) = (z_*,t_*+\delta(t))$, where $\delta:\R \to \R$ verifies $\delta'(t)>\delta_0>0$, and there exists $j \in \{0,1,\dots,k-1\}$ such that $\tilde\eta(T) = (z_*,t_*+\Delta) \in L_j$ for some $\Delta>0$.
If $j=0$ then $c_0$ is not injective, contradicting the assumption made above.
Assume $j\neq0$.
One finds $s',s''\in\R$ such that $c_0(s') = z_* = c_j(s'')$ from where it follows that $$ \begin{aligned} & h^{mk}(c_0(s'+mk)) = c_0(s') = z_* = c_j(s'') = h^{mk}(c_j(s''+mk)) \\ & \qquad \Rightarrow c_0(s'+mk) = c_j(s''+mk) \quad \forall m \in \N. \end{aligned} $$
We conclude that the injective immersions $c_0$ and $c_j$ meet at two distinct points. 
The existence of the desired loop contained in $C$ follows.
\end{proof}

In order to find a contradiction, we now consider the iterates $h^n(\gamma) \subset S$, where~$\gamma$ is the loop obtained from Lemma~\ref{lemma_injective}.
A point in $p \in \gamma \cap h^n(\gamma)$ gives rise to a Reeb chord of $L$ provided $n$ is larger than a constant that depends only on $\gamma$.
Moreover, if there exists a sequence $n_j \in \N$ such that $n_j \to +\infty$ and $\gamma \cap h^{n_j}(\gamma) \neq \emptyset$ then there are infinitely many Reeb chords of $L$.

If $\gamma$ bounds a disk in $S$, then for all $n$ the loop $h^n(\gamma)$ also bounds a disk in $S$.
By Stokes' theorem the disks bounded by $\gamma$ and $h^n(\gamma)$ have the same $d\lambda$-area.
Since $S$ has finite area, for every $N\geq0$ there exists $n\geq N$ such that $\gamma$ and $h^n(\gamma)$ have non-empty intersection.
By the above there exists an infinite sequence of integers~$n_i$, $i\in \N$, such that $n_i \to \infty$ and $\gamma \cap h^{n_i}(\gamma) \neq \emptyset$.

Assume now that $\gamma$ does not bound a disk.
Since $\gamma$ is embedded, this implies that $\gamma$ is not contractible.
We claim that there exists a sequence $n_i\to +\infty$ such that the embedded loops $\gamma$ and $h^{n_i}(\gamma)$ intersect or are parallel.
If not then one finds $N \geq 1$ so that if $n\geq N$ then $\gamma$ and $h^n(\gamma)$ do not intersect and are not parallel.
Applying powers of $h$ we see that the sequence $\{h^{kN}(\gamma)\}_{k\geq1}$ consists of loops that are mutually non-intersecting and mutually non-parallel.
But this is impossible since $S$ is a compact surface.
Hence the desired sequence $n_i$ exists.
The loops $\gamma$ and $h^{n_i}(\gamma)$ must intersect when they are parallel because $h$ has flux zero.
Hence $\gamma$ intersects inifinitely many iterates of $\gamma$ under $h$, as desired.

This proves the first part of Theorem~\ref{thm: chords}.

\medskip

We now assume that we have a $\partial$-strong Birkhoff section $S$.
The Legendrian knot $L$ has infinitely many chords, and we now need to prove that there are infinitely many  chords whose interiors are disjoint.
Assume, by contradiction, that there are only finitely many Reeb chords with disjoint interiors.
Then there is a finite number of periodic orbits intersecting $L$, that is containing chords. 
Up to taking a power of the first return map $h$ on $int(S)$, all these periodic orbits correspond to a finite collection $P$ of fixed points. 
Consider the intersection with $S\setminus P$ of all the orbits intersecting $L$. This defines curves in~$S$.
The strategy of the proof is to study the intersection of these curves with their iterates under the return map $h$.

We address first the case in which $L$ does not intersect $\partial S$, for which we prove that there are infinitely many geometrically distinct Reeb chords. 
Consider the arc $\gamma$ given by Lemma~\ref{lemma_injective}.
We claim that there is a curve $\hat\gamma:[0,1]\to S$, 
whose image is contained in $\gamma$ such that  $\hat\gamma|_{[0,1)}$ is injective and $\hat\gamma|_{(0,1)}$ is contained in $\gamma\setminus P$.  Moreover, the endpoints of $\hat\gamma$ either coincide or are in $P$.
In fact, there are two possibilities:\begin{enumerate}
\item If $\gamma\cap P=\emptyset$ or $\gamma\cap P=\gamma(0)=\gamma(1)$, then we take $\hat\gamma=\gamma$ and the point $\hat\gamma(0)=\hat\gamma(1)$ may or may not be in $P$.
\item If $\gamma\cap P$ is not as above, we take $\hat\gamma$ that coincides with a component of $\gamma\setminus P$. 
In this case, $\hat\gamma(0)\neq \hat\gamma(1)$ are both in $P$.
\end{enumerate}

We start with case $(1)$, we continue to assume that there are finitely many distinct Reeb chords and deduce a contradiction.
As before one argues to find a sequence $n_i \to \infty$ such that $\hat{\gamma} \cap h^{n_i}(\hat{\gamma}) \neq \emptyset$.
Assume first that $\hat\gamma$ is contractible.
If $\# \hat\gamma \cap h^{n_i}(\hat\gamma) = 1$ for some $i$ then, since both loops are contractible, one obtains a pinched annulus bounded by $\hat\gamma \cup h^{n_i}(\hat\gamma)$, in contradiction to the fact that $h^{n_i}$ has zero flux.
This shows that $\# \hat\gamma \cap h^{n_i}(\hat\gamma) \geq 2$ for all $i$, in particular, the set $\hat\gamma \cap h^{n_i}(\hat\gamma)$ contains a point in $S \setminus P$ for every $i$, which gives infinitely many geometrically distinct chords for $L$ and the desired contradiction.
If now $\hat\gamma$ is not contractible and $\hat\gamma \subset S\setminus P$ then $\hat\gamma \cap h^{n_i}(\hat\gamma) \subset S\setminus P$ for all $i$. 
We get infinitely many geometrically distinct Reeb chords and again a contradiction.
Finally to conclude case $(1)$ we need to handle the case where $\hat\gamma$ is not contractible and $\hat\gamma \cap P = \hat\gamma(0) = \hat\gamma(1)$.
Some of the previous arguments need to be repeated.
We need to find $n_i \to +\infty$ such that $(\hat\gamma \cap h^{n_i}(\hat\gamma)) \setminus P \neq \emptyset$.
This will give infinitely many geometrically distinct Reeb chords and the desired contradiction.
If such a sequence $n_i$ does not exist then we find $N$ such that $n \geq N \Rightarrow \hat\gamma \cap h^n(\hat\gamma) = \hat\gamma(0)$.
Consider the sequence $\Gamma_k = h^{kN}(\hat\gamma)$, $k\geq 0$.
We see each $\Gamma_k$ as a loop based at $\hat\gamma(0)$.
All these loops meet only at $\hat\gamma(0)$, and are mutually non-parallel in the sense that they are not homotopic to each other modulo $\hat\gamma(0)$.
In fact, if for $0\leq k<l$ the loop $\Gamma_k$ intersects $\Gamma_l$ at some point $q \neq \hat\gamma(0)$, then $\hat\gamma = h^{-kN}(\Gamma_k)$ intersects $\Gamma_{l-k} = h^{-kN}(\Gamma_l) = h^{(l-k)N}(\hat{\gamma})$ at $h^{-kN}(q) \neq \hat\gamma(0)$, which is impossible.
Here the fact that $\hat\gamma(0)$ is a fixed point of $h$ was used.
Moreover, if $\Gamma_k$ is homotopic to $\Gamma_l = h^{(l-k)N}(\Gamma_k)$ modulo $\hat\gamma(0)$ then $\Gamma_k \cup \Gamma_l$ bounds a pinched annulus, in contradiction to the fact that $h^{(l-k)N}$ has zero flux.
The existence of loops $\Gamma_k$ with the above properties is in contradiction to the fact that $S$ is a compact surface and, as such, has finite genus and finitely many boundary components.

Now we consider case $(2)$. 
Since $\hat{\gamma}(0)$ and $\hat{\gamma}(1)$ are fixed points of $h$, there exists $N\in\N$ such that $n\geq N \Rightarrow \hat{\gamma} \cap h^n(\hat{\gamma}) = \{\hat{\gamma}(0),\hat{\gamma}(1)\}$.
Otherwise one finds infinitely many geometrically distinct Reeb chords for $L$.
Now, the arcs $\{h^{mN}(\hat{\gamma})\}_{m\in\N}$ mutually intersect only at end points. Since $S$ has finite genus and finitely many boundary components, at least two of them are parallel, in the sense that they are homotopic modulo end points, and disjoint. Thus there exists $l\in\N$ such that $\hat{\gamma}$ and $h^l(\hat{\gamma})$ are parallel.

Since $\hat{\gamma}$ and $h^l(\hat{\gamma})$ are parallel embedded arcs meeting only at end points, the set $\hat{\gamma} \cup h^l(\hat{\gamma})$ is the boundary of a disk $D_1 \subset S$ that we call a bi-gon.

Define $D_m = h^l(D_{m-1})=h^{(m-1)l}(D_1)$ for each $m\geq 1$, and observe that for every $m$ the bi-gons $D_m$ and $D_{m+1}$ intersect along $h^{ml}(\hat{\gamma})$.
Then $D_m$ is a disk bounded by $h^{(m-1)l}(\hat{\gamma}) \cup h^{ml}(\hat{\gamma})$.

We prove the following by induction on $\kappa \geq 1$: $E_\kappa = D_1 \cup \dots \cup D_\kappa$ is a disk in $int(S)$ with boundary equal to $\hat{\gamma} \cup h^{\kappa l}(\hat{\gamma})$, and the interior of the disks $D_1,\dots,D_\kappa$ are mutually disjoint.
When $\kappa=1$ there is nothing to be proved.
Assuming that it holds for $\kappa$, we now prove the claim for $\kappa+1$.
Since $E_\kappa$ is a disk, it divides $S$ in two connected components, one of them being the interior $int(E_\kappa)$ of $E_\kappa$.
To prove the that the claim holds for $\kappa+1$ it suffices to show that the arc $h^{(\kappa+1)l}(\hat{\gamma})$ is contained $S\setminus int(E_\kappa)$.
If not then $h^{(\kappa+1)l}(\hat{\gamma}) \setminus \{\hat{\gamma}(0),\hat{\gamma}(1)\}$ is contained in $int(E_\kappa)$ and does not intersect the arcs $h^l(\hat{\gamma}),\dots,h^{(\kappa-1)l}(\hat{\gamma})$. Then $E_{\kappa +1}$ covers $S$, the collection of arcs $h^l(\hat{\gamma}),\dots,h^{\kappa l}(\hat{\gamma})$ divides $S$ into bi-gons 
and $S$ is a sphere, contradicting the fact that $S$ has boundary. Thus $h^{(\kappa+1)l}(\hat{\gamma})$ is contained $S\setminus int(E_\kappa)$ and the interior of $D_{\kappa+1}$ is disjoint from $E_\kappa$. 
The induction argument is finished and the claim is proved.
As a consequence, for every $\kappa \in \N$ the $d\lambda$-area of $S$ is larger than $\kappa a$ where $a$ is the $d\lambda$-area of $D_1$, in contradiction to the finiteness of the total $d\lambda$-area of $S$.

We have so far proved that $L$ admits infinitely many geometrically distinct Reeb chords provided that $L$ does not intersect $\partial S$.
The missing case is when $L$ intersects~$\partial S$.
We treat this case now.
Here we will use the hypothesis that the Birkhoff section $S$ is $\partial$-strong.

We proceed with the standing assumption that there are finitely many geometrically distinct Reeb chords for~$L$.
From this assumption we shall derive a contradiction except when all of the following hold: $S$ is a disk or an annulus, the Reeb flow has exactly two periodic orbits which form the link $\partial S$, and $L$ intersects both periodic Reeb orbits.

Consider the manifold $M_{\partial S}$ obtained by blowing up the boundary of $S$.
The Reeb vector field extends from $M \setminus S = M_{\partial S} \setminus \partial M_{\partial S}$ to $M_{\partial S}$. 
Observe that $M_{\partial S}$ is a 3-manifold with boundary where the boundary consists of a torus for each connected component of $\partial S$. 
The extension of the Reeb vector field is smooth, non-singular and tangent to~$\partial M_{\partial S}$. 
The surface $S$ defines a collection of curves in~$\partial M_{\partial S}$. 
The extended Reeb vector field is transverse to these curves and has a well-defined first return map on them. 
Moreover, the return map extends from $int(S)$ to $S$.
This map is again denoted by $h$.
Note that the extended Reeb vector field is no longer Reeb, but the map $h$ keeps the flux zero property for the $2$-form~$d\lambda$, even though the latter is vanishing along the boundary.

Similarly as before $M_{\partial S} = S \times \R / \sim$ where $(z,t) \sim (h(z),t-1)$. 
The quotient projection $S \times \R \to M_{\partial S}$ defines an infinite-cyclic cover.
One can arrange so that the flow of the extended Reeb vector field lifts to a positive multiple of $\frac{\partial}{\partial t}$ where $t$ denotes the $\R$-coordinate in $S\times\R$.
The Legendrian knot $L$ gives a collection of Legendrian arcs in $M_{\partial S}$ whose endpoints are in $\partial M_{\partial S}$. 
Note that these arcs meet $\partial M_{\partial S}$ transversely since $L$ intersects $\partial S$ non-tangentially ($L$ is Legendrian and $\partial S$ is transverse to the contact structure).
Choose one of these arcs and denote it by $L_0$.
Choose also a lift $L_0'$ of $L_0$ to $S\times \R$.
Let $\delta$ be the projection of $L_0'$ onto $S$ by the projection $S\times \R \to S$.
Then $\delta$ is an immersed arc whose endpoints are in $\partial S$.
As before $P \subset int(S)$ denotes the set of interior periodic points of $h$ corresponding to the periodic Reeb orbits in $M\setminus \partial S$ involved in the finitely many geometrically distinct Reeb chords for~$L$.
Up to taking a power of $h$ we can assume that $h$ fixes every point of $P$.

If $\delta$ has an interior double point, i.e. a self-intersection point away from endpoints, then we find a loop $\gamma\subset \delta$ contained in the interior of $S$.
We are back to the situation handled above, and argue as before to conclude that $L$ admits infinitely many geometrically distinct Reeb chords.
This is a contradiction.

We have then that $\delta$ has no interior double point. We consider two cases depending on whether $\delta$ instersects $P$ or not.

Assume that $\delta\cap P=\emptyset$.

Up to taking a power of $h$, we can assume that $h$ leaves each connected component of $\partial S$ invariant.
Then all the arcs $h^k(\delta)$ have endpoints and in the same connected components of $\partial S$ as the corresponding endpoints of $\delta$.
There exists $N \in \N$ such that $n\geq N$ implies $\delta$ and $h^n(\delta)$ have no interior intersection points.
If such $N$ does not exist we get infinitely many geometrically distinct Reeb chords for $L$, in contradiction to the standing assumption.
Up to replacing $h$ by $h^N$ we can assume that for every $0\leq k<l$ the arcs $h^k(\delta)$ and $h^l(\delta)$ have no interior intersection points.
Since $S$ has finite genus and finitely many boundary components, there exist $k<l$ such that $h^k(\delta)$ and $h^l(\delta)$ are parallel.
We find $N_1=l-k$ such that $\delta$ and $h^{N_1}(\delta)$ are parallel. 
Replacing $h$ by $h^{N_1}$ one can assume that $\delta$ and $h(\delta)$ are parallel.
Then there is a disk $G\subset S$ whose boundary consists of $\delta$ and $h(\delta)$ intercalated by (possibly degenerate) arcs in $\partial S$.
We call such an object a $4$-gon.
Thus $h(\delta)$ and $h^2(\delta)$ are pieces of the boundary of the $4$-gon $h(G)$ and so on. 
Since we assume there are no interior intersection points between arcs $h^k(\delta)$ and $h^l(\delta)$ for all~$0\leq k<l$, and since $h$ preserves area, there must be a minimal integer $k\geq 2$ such that $h^k (\delta)$ is contained in~$G$. 
Then $S$ is an annulus decomposed into $4$-gons and $L$ intersects the two components of $\partial S$. 

We now show that $h$ has no periodic points in $int(S)$.
 
 If we consider the sequence of $N+1$ essential parallel and disjoint arcs $\delta,\dots,h^N(\delta)$ in the annulus $S$, then at least two of them bound a quadrilateral in $S$ whose $d\lambda$-area is less than $\frac{1}{N} area (S)$. 
Thus, for every $\epsilon>0$ and every $p\geq1$ there exists a $4$-gon $G'$ of $d\lambda$-area less than $\epsilon$, bounded by arcs of the form $h^n(\delta)$, $h^m(\delta)$ for some $n,m\in\N$ intercalated by (possibly degenerate) arcs in $\partial S$, such that $G',h^{n-m}(G'),\dots,h^{p(n-m)}(G')$ have mutually disjoint interiors, and the sequence $(h^{(n-m)i} (G'))_{i\in \Z}$ covers $S$.

Observe that since $\gamma\cap P=\emptyset$ the curve $\gamma$ contains no periodic point. Assume by contradiction that $x \in int(S)$ is a periodic point of $h$. Let $p$ be its minimal period.
By the above claim there exists a $4$-gon $G'$ such that $x\in int(G')$ and $int(G') \cap  int(h^{p(n-m)}(G')) = \emptyset$.
It follows from $ h^{p(n-m)}(x) = x$ that $x$ belongs to the intersection of the interiors of $G'$ and $h^{p(n-m)}(G')$, again a contradiction.

We have been denoting by $h$ an iterate of the original return map to $S$, and we proved that this iterate does not have periodic points in $int(S)$.
Hence also the original return map has no periodic points in $int(S)$, which means that the only periodic Reeb orbits are those forming~$\partial S$.
These must be geometrically distinct by the result from~\cite{CGH}.
Hence the Reeb flow has exactly two periodic orbits.
By the result from~\cite{CGHHL} the map $h$ is an irrational pseudo-rotation and the two boundary orbits are non-degenerate elliptic.
The proof in the case $\delta \cap P = \emptyset$ is complete.

Consider now the final case where $\delta$ has no interior double points and $\delta\cap P\neq \emptyset$.
Inside $\delta$ there exists a subarc $\delta'$ with an endpoint in $\partial S$, an endpoint at some $x_0 \in P$ and such that $\delta'\cap  P=\{x_0\}$.
One argues similarly as in the previous case with $\delta'$ in the place of $\delta$ and tri-gons in the place of 4-gons, to deduce that $S$ is a disk.
Let us describe the main steps.

Up to replacing $h$ by a higher iterate, $x_0$ is a fixed point, each component of $\partial S$ is invariant by $h$ and
$h^k(\delta')$ and $h^l(\delta')$ have no interior intersections, for all $0 \leq k < l$. As before, we deduce the existence of 
$N_1$ such that $\delta'$ and $h^{N_1}(\delta')$ are parallel.
Replacing $h$ by $h^{N_1}$, the arcs $\delta'$ and $h(\delta')$ are parallel.
There exists a disk $D$ whose boundary consists of $\delta'$, $h(\delta')$ and a (possibly degenerate) arc in $\partial S$.
We call such an object a tri-gon.
Note that $h(D)$ is again a $3$-gon sharing $h(\delta')$ with $D$.
Since $h$ preserves $d\lambda$-area and $h^k(\delta')$, $h^l(\delta')$ have no interior intersection points for every $0\leq k<l$, there must be a minimal $k$ such that $h^k(\delta') \subset D$. Thus $S$ is a disk, and $L$ intersects both $\partial S$ and the periodic Reeb orbit through $x_0$. 

Analogously to the case in which $S$ is an annulus, we deduce that $h$ has no periodic points in $int(S)\setminus \{x_0\}$. By the result from~\cite{CGHHL} the map $h$ is an irrational pseudo-rotation and the two boundary orbits are non-degenerate elliptic. This completes the proof of Theorem~\ref{thm: chords}.
\end{proof}

\appendix

\section{Embedded and thin Fried's pairs of pants}\label{appendix}

\centerline{In collaboration with  Pierre Dehornoy\footnote{PD address: Aix-Marseille Université, CNRS, I2M, Marseille, France, \email{pierre.dehornoy@univ-amu.fr}, URL: https://www.i2m.univ-amu.fr/perso/pierre.dehornoy/}.}

\bigskip

Given vector field on a $3$-manifold and a hyperbolic periodic orbit with homoclinics in opposite quadrants, Fried gave a construction of an immersed partial section that is a pair of pants and whose interior is transverse to the vector field~\cite{friedanosov}. 
This construction was used and explained in subsequent works~\cite{CDR, CM} (also in Section~\ref{sec: constraint}). 
In the proof of Theorem~\ref{thm: BS1}, in Section~\ref{sec: legendrian}, we need a slightly refined version, where the pair of pants is embedded and small in a precise sense. 
The proof we give follows the same track, with some extra care when choosing the right rectangles and iterates to ensure the two extra properties. 
In particular we follow some notations from~\cite{CM}.

Recall that an embedded partial section for a flow is an embedded surface bounded by a finite number of periodic orbits and whose interior is transverse to the considered flow. 

\begin{proposition}
\label{lemma: embedded}
Let $M$ be a compact 3-manifold equipped with a Riemannian metric, $\phi^t$ a non-singular flow on~$M$, $\gamma$ a hyperbolic periodic orbit of~$\phi^t$ that is neither a sink nor a source. 
Assume that $\gamma$ admits two homoclinics~$h_1, h_2$ in opposite quadrants.
Then for every~$\epsilon>0$, there exists a pair of pants~$P$ that is an embedded partial section for~$\phi^t$ contained in an~$\epsilon$-neighborhood of~$h_1\cup h_2$.
Moreover, given a flow-box~$B$ of size~$\epsilon$ centered on a point of~$h_1\cup h_2$ and not intersecting~$\gamma$, the surface $P$ intersects $B$ along one rectangle only.

If $\phi^t$ is a Reeb flow of a contact structure $\xi$, the characteristic foliation $\xi P$ of $P$ is {\it short}: there exists $C>0$ (independent of $P$) such that for every $x\in P$, the leaf $l_x$ of  $\xi P$ through $x$ is contained in a ball of radius $C\epsilon$.
\end{proposition}

\begin{proof}
Fix a point~$z_0$ on~$\gamma$, and consider an oriented transverse disc~$D$ containing~$z_0$ (see Figure~\ref{fig: embpants1}).
Call $\psi$ the first-return map on the part of~$D$ where the first-return time is close to the period of~$\gamma$. 
Assume~$D$ small enough so that~$D\cap \gamma=\{z_0\}$, that $\psi$ is uniformly hyperbolic, that $D$ is a rectangle bounded by two arcs in~$W^s(\gamma)$ and two arcs in $W^u(\gamma)$, and that $D$ has diameter smaller than~$\epsilon$.
Denote by~$\ell^s$ and~$\ell^u$ the path-connected component of~$W^s(\gamma)\cap D$ and $W^u(\gamma)\cap D$ containing~$z_0$. 

Let $z_1$ denote a point of~$\ell^u \cap h_1$. 
Since $h_1$ is a homoclinic, there exists a first-return time $T_1>0$ such that $\phi^{T_1}(z_1)$ is in~$\ell^s\cap h_1$. 
Consider a neighborhood $U_1$ of~$z_1$ in~$D$. 
At the expense of reducing~$U_1$, one can assume that for two small numbers $\delta\ge 0$ and $\epsilon\ge 0$ we have that:
\begin{itemize}
    \item $U_1$ is bounded by arcs in~$W^s(\gamma)\cap D$ and $W^u(\gamma)\cap D$;
    \item the first-return time from~$U_1$ to $D$, denoted by $\tau_1:U_1\to \R_+$, is continuous with image in~$[T_1-\delta, T_1+\delta]$;
    \item the map $\phi^{\cdot}(\cdot):[0,T_1+\delta]\times U_1\to M$ is an embedding whose image is contained in a tubular neighborhood of $h_1$ of size~$\epsilon$.
\end{itemize}  

Denote by $U'_1$ the image of~$U_1$ under the first-return map on~$D$, namely $U'_1=\phi^{\tau_1(\cdot)}(\cdot)(U_1)$. 
We can take $\epsilon$ small enough so that the last condition above implies  $U'_1\cap U_1=\emptyset$. 
At the expense of another reduction, one can also assume that 
$$\psi(U'_1) \cap U'_1=\emptyset \qquad \mbox{and} \qquad \psi^{-1}(U_1) \cap U_1=\emptyset.$$

Playing the same game with the homoclinic $h_2$ yields a point~$z_2\in\ell^u \cap h_2$ such that $z_1$ and $z_2$ are not in the same connected component of $\ell^u\setminus \{z_0\}$. We also recover: 
\begin{itemize}
    \item a neighborhood~$U_2$ of~$z_2$ in~$D$;
    \item a continuous first-return time function~$\tau_2:U_2\to\R_+$ so that $\phi^{\tau_2(z_2)}(z_2)$ is in~$\ell^s\cap h_2$;
    \item an embedded tube obtained by flowing~$U_2$ contained in a tubular neighborhood of $h_2$ of size~$\epsilon$;
    \item the set $U_2'$ such that $\psi(U'_2) \cap U'_2$ and $\psi^{-1}(U_2) \cap U_2$ are empty. 
\end{itemize} 

Since $\psi$ is uniformly hyperbolic on~$D$, there exist minimal integers $n_{1,1}, n_{1,2}>0$ such that $\psi^{n_{1,1}}(U'_1)\cap U_1$ and $\psi^{n_{1,2}}(U'_2)\cap U_1$ are non-empty. 
Note that reducing~$U_1$ in the unstable direction on one or the other side may increase either~$n_{1,1}$ or $n_{1,2}$ independently. 
Therefore, at the expense of another reduction of~$U_1$ in the unstable direction, one can ensure $n_{1,1}=n_{1,2}$. 

Similarly,  there exists minimal integers $n_{2,1}, n_{2,2}>0$ such that $\psi^{n_{2,1}}(U'_1)\cap U_2$ and $\psi^{n_{2,2}}(U'_2)\cap U_2$ are non-empty. 
And at thxe expense of reducing~$U_2$ in the unstable direction, one may enforce~$n_{2,1} = n_{2,2}$.

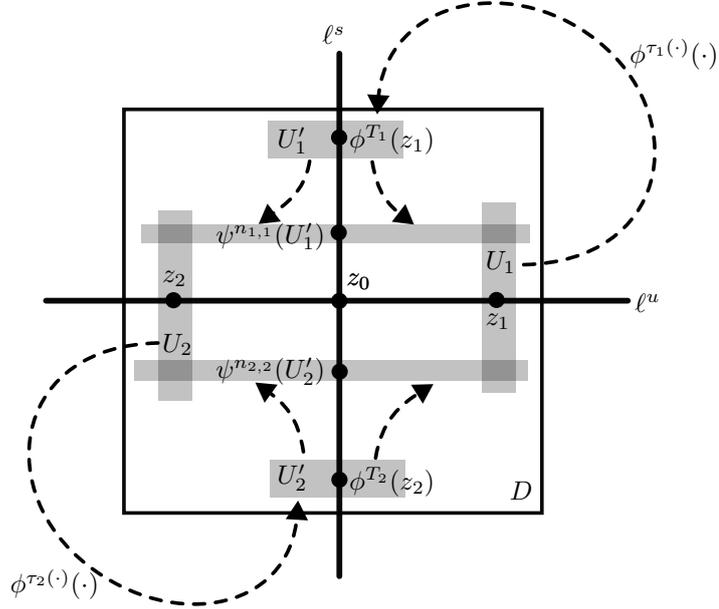
\begin{figure}[h]
\includegraphics[width=.7\textwidth]{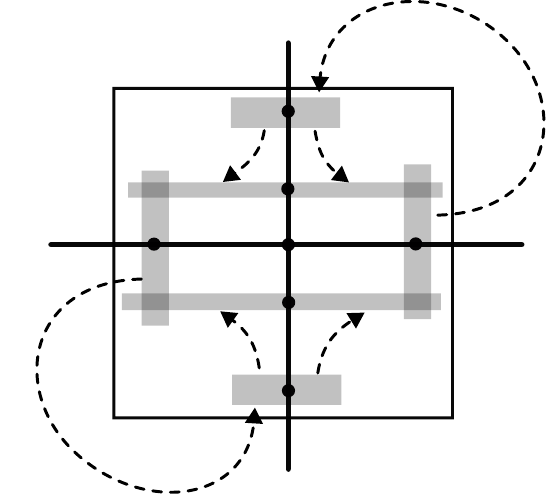}
\put(-55,40){$D$}
\put(-116,121){$z_0$}
\put(-8,111){$\ell^u$}
\put(-125,212){$\ell^s$}
\put(-116,121){$z_0$}
\put(-64,106){$z_1$}
\put(-64,127){$U_1$}
\put(-10,205){$\phi^{\tau_1(\cdot)}(\cdot)$}
\put(-142,173){$U'_1$}
\put(-115,173){$\phi^{T_1}(z_1)$}
\put(-165,137){$\psi^{n_{1,1}}(U'_1)$}
\put(-185,122){$z_2$}
\put(-185,96){$U_2$}
\put(-142,46){$U'_2$}
\put(-115,44){$\phi^{T_2}(z_2)$}
\put(-165,86){$\psi^{n_{2,2}}(U'_2)$}
\put(-242,5){$\phi^{\tau_2(\cdot)}(\cdot)$}
\caption{The dynamics of the first-return map along~$\phi^t$ on a transverse disc~$D$ around a hyperbolic periodic point~$z_0$.}
\label{fig: embpants1}
\end{figure}

A classical argument then implies the existence of a unique fixed point~$w_{1,1}$ for the first-return map from $U_1$ to  $\psi^{n_{1,1}}(U'_1)$ along the flow, of a unique fixed point~$w_{2,2}$ for the first-return map from $U_2$ to $\psi^{n_{2,2}}(U'_2)$, and of a period two orbit~$(w_{1,2}, w_{2,1})$ where $w_{1,2}$ is in $U_1\cap \psi^{n_{2,2}}(U'_2)$ and $w_{2,1}$ is in $U_2\cap \psi^{n_{1,1}}(U'_1)$. These points correspond to three periodic orbits of flow.

Denote by $\alpha_{1,1}$ a segment in~$U_1$ connecting $w_{1,1}$ to~$w_{1,2}$ that is transverse to both~$W^s(\gamma)\cap D$ and~$W^u(\gamma)\cap D$ (see Figure~\ref{fig: embpants2}).
When pushed along the flow, $\alpha_{1,1}$ produces a band  tangent to the vector field that we flow up to the first return  to $D$. By the properties above, $R_1$ lies in the embedded flow box~$\phi^\cdot(\cdot):[0,T_1+\delta]\times U_1\to M$ that is contained in the~$\epsilon$-tubular neighborhood of $h_1$. Pushing further with the flow, the obtained band $R_1$ lies in the $\epsilon$-neighborhood of~$h_1$ going from~$U'_1$ to~$\psi^{n_{1,1}}(U'_1)$. 
We denote by~$\alpha_{1,2}$ the image of~$\alpha_{1,1}$ in~$\psi^{n_{1,1}}(U'_1)$ under the flow. 
By construction, it connects~$w_{1,1}$ to~$w_{2,1}$. 
The description above implies that the ribbon~$R_1$ connecting~$\alpha_{1,1}$ to~$\alpha_{1,2}$ along the flow is embedded and lies in a tubular neighborhood of $h_1$ of diameter~$\epsilon$. 

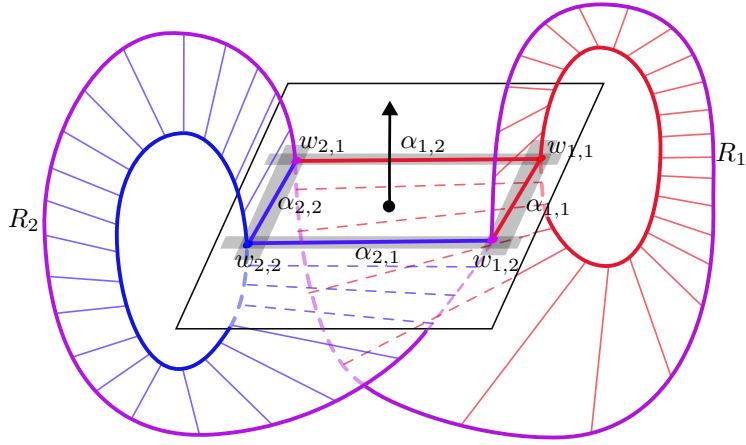
\begin{figure}[h]
\includegraphics[width=.7\textwidth]{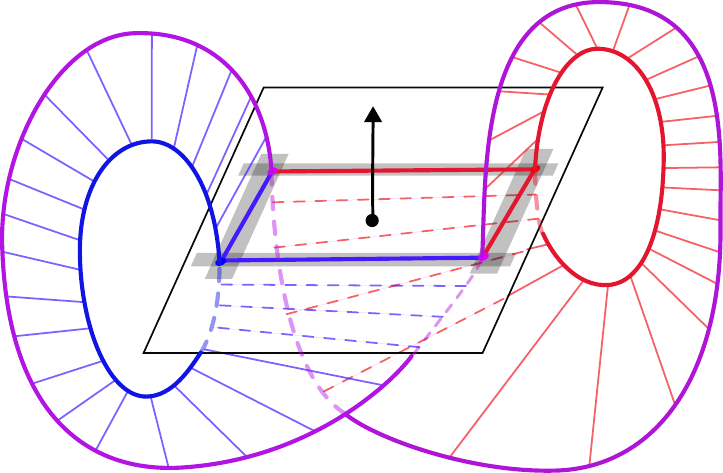}
\put(-63,108){$w_{1,1}$}
\put(-91,65){$w_{1,2}$}
\put(-156,110){$w_{2,1}$}
\put(-180,65){$w_{2,2}$}
\put(-71,86){$\alpha_{1,1}$}
\put(-118,110){$\alpha_{1,2}$}
\put(-135,68){$\alpha_{2,1}$}
\put(-164,87){$\alpha_{2,2}$}
\put(0,105){$R_1$}
\put(-265,80){$R_2$}
\caption{Construction of the pair of pants~$P$ from the three periodic points $w_{1,1}, w_{1,2}, w_{2,2}$ in~$D$.}
\label{fig: embpants2}
\end{figure}

Similarly, we obtain a segment $\alpha_{2,2}$ in~$U_2$ that connects $w_{2,2}$ to~$w_{2,1}$, which is pushed along the flow to a segment~$\alpha_{2,1}$ in~$\psi^{n_{2,2}}(U'_2)$ that connects~$w_{2,2}$ to~$w_{1,2}$, and such that the ribbon~$R_2$ connecting~$\alpha_{2,2}$ to~$\alpha_{2,1}$ along the flow is embedded and lies in a tubular neighborhood of $h_2$ of diameter~$\epsilon$.  

Finally we consider the quadrilateral~$Q$ in~$D$ bounded by the four segments~$\alpha_{1,1}$, $\alpha_{1,2}$, $\alpha_{2,2}$ and $\alpha_{2,1}$. 
It is transverse to the flow. 
The union~$Q\cup R_1\cup R_2$ is a topological surface bounded by the three periodic orbits through~$w_{1,1}, w_{1,2}, w_{2,2}$. 
It can be smoothed by an arbitrary small isotopy relative to its boundary into an embedded pair of pants~$P$ that is transverse to the flow. 
By construction, $P$ lies an $\epsilon$-neighborhood of~$h_1\cup h_2$.

It remains to see that the characteristic foliation of $P$ is ``short".
We know that $R_i$, $i=1,2$, is contained in an $\epsilon$-neighborhood of $h_i$. We get, $R_i \subset D(\epsilon)\times [0,T]_t$, where $\frac{\partial}{\partial t} =R$ and the disk $D(\epsilon)\times \{t\}$ is of size $\epsilon$ for all $t\in [0,T]$, and in particular is of $d\lambda$-area $\leq \pi \epsilon^2$.
Moreover, we can make sure that there is a constant $\theta_0$ such that the contact plane $\xi$ makes an angle $<\theta_0$ with $T(D(\epsilon)\times \{t\})$ for all $t\in [0,T]$.
This last condition implies that there is a constant $C_0$ such that for every 
straight line $l_{p,q}$ between points $p,q \in D(\epsilon)$, its Legendrian lifts in $D(\epsilon)\times [0,T]_t$ are of the form $(p,t)$ and $(q,t')$ with $\vert t-t'\vert <C_0 \epsilon$.

We take $x\in P$ who is at intrinsic distance $>(C_0+2\pi\epsilon)\epsilon=C\epsilon$ from $Q$ in $P$. We assume without loss of generality that $x\in R_1$. There is a flow box $D(\epsilon)\times [A-C\epsilon,A+C\epsilon]_t \subset D(\epsilon)\times [0,T]_t$ containing $x$ at $\{ t=A\}$, where  the connected component of $P \cap (D(\epsilon)\times [A-C\epsilon,A+C\epsilon]_t)$ that contains $x$ is contained in $R_1$ and is $t$-invariant.

Now we claim that the leaf $l_x$ of the characteristic foliation $\xi P$ that passes through $x$ is contained in a ball of radius $C\epsilon$.
Indeed, if $y\in l_x$ and $l_{xy}$ denotes the subarc of $l_x$ between $x$ and $y$, and if $p:D(\epsilon)\times [A-C\epsilon,A+C\epsilon]_t \to D(\epsilon)$ is the projection along $R=\frac{\partial}{\partial t}$, then we can take a straight line $d$ in $D(\epsilon)$ between $p(x)$ and $p(y)$. The Legendrian lift of $d$ starting at $x$ ends in a point $y'$ that is $C_0 \epsilon$-close to $x$.
Now, the difference of $t$-coordinates between $y$ and $y'$ in  $D(\epsilon)\times [A-C\epsilon,A+C\epsilon]_t$ is given, by Stokes' formula, by the signed $d\lambda$-area between $d$ and $l_{xy}$ in $D(x,\epsilon)$, which is $<2\pi \epsilon^2$.
Thus the conclusion.

If $x$ is in a $C\epsilon$-neighborhood of $Q$, then its leaf either stays in this neighborhood, hence is short, or exits this neighborhood and we can apply the previous considerations.
\end{proof}


\begin{thebibliography}{99}

\bibitem{AA} V. I. Arnol'd, A. Avez,  {\it Probl\`emes ergodiques de la m\'ecanique classique.} 
Monographies internationales de math\'ematiques modernes. 9. Paris: Gauthier-Villars. 243 p. (1967). 


\bibitem{AM} M. R. R. Alves, M. Mazzuchelli {\it From curve shortening to flat link stability and Birkhoff sections of geodesic flows }, arXiv:2408.11938.

 







\bibitem{CDHR} V. Colin, P. Dehornoy, U. Hryniewicz and A. Rechtman, \textit{Generic properties of 3-dimensional Reeb flows: Birkhoff sections and entropy.} Comment. Math. Helv. 99, No. 3, 557-611 (2024). 

\bibitem{CDR} V. Colin, P. Dehornoy and A. Rechtman, \textit{On the existence of supporting broken book decompositions for contact forms in dimension $3$.} Invent. math. 231, 1489--1539 (2023).


\bibitem{C_Annals} G. Contreras, \textit{Geodesic flows with positive topological entropy, twist maps and hyperbolicity.} Annals of Math., 172 (2010), 761--808.

\bibitem{CM} G. Contreras and M. Mazzucchelli, \textit{Existence of Birkhoff sections for Kupka-Smale Reeb flows of closed contact 3-manifolds.} Geom. Funct. Anal. 32, No. 5, 951-979 (2022). 

\bibitem{CO_geod} G. Contreras and F. Oliveira, \textit{Birkhoff Program for Geodesic Flows of Surfaces and Applications: Homoclinics.} J. Dyn. Diff. Equat. (2024). 

\bibitem{CO_primeends} G. Contreras and F. Oliveira, \textit{No elliptic points from fixed prime ends}, \texttt{arXiv:2205.14768}.

\bibitem{CP} G. Contreras and G. Paternain, \textit{Genericity of geodesic flows with positive topological entropy on $S^2$}, J. Differential Geom. {\bf 61}, No. 1, 1--49 (2002). 

\bibitem{CGH} D. Cristofaro-Gardiner and M. Hutchings, \textit{From one Reeb orbit to two}, J. Differential Geom. {\bf 102}, No.~1, 25--36 (2016).

\bibitem{CGHHL} D. Cristofaro-Gardiner, U. Hryniewicz, M. Hutchings and H. Liu, \emph{Contact three-manifolds with exactly two simple Reeb orbits.} Geom. Topol. 27:9 (2023) 3801--3831.

\bibitem{CGHHL2}
 D. Cristofaro-Gardiner, U. Hryniewicz, M. Hutchings, H. Liu, \emph{Proof of Hofer-Wysocki-Zehnder's two or infinity conjecture,} arxiv:231007636.



%
%
%
%
\bibitem{friedanosov} D. Fried, {\it Transitive Anosov flows and pseudo-Anosov maps}, {Topology}, {22} (1983), 299--303.
%

%

\bibitem{Hall} M. Hall, \textit{A topology for free groups and related topics.} Ann. of Math. 52(1950), 127--139.

%
%
%
%
%
%
%
%

\bibitem{HT1} M. Hutchings and C. H. Taubes, {\it  Proof of the Arnold chord conjecture in three dimensions I}, Math. Res. Lett. {\bf 18} (2011), 295-313. 

\bibitem{HT2} M. Hutchings and C. H. Taubes, {\it Proof of the Arnold chord conjecture in three dimensions II},
Geometry and Topology {\bf 17} (2013), 2601-2688.
 
\bibitem{Irie_0} K. Irie, \textit{ Dense existence of periodic Reeb orbits and ECH spectral invariants.} J. Mod. Dyn. 9 (2015), 357--363. 

\bibitem{irie} K. Irie, \textit{Equidistributed periodic orbits of $C^{\infty}$-generic three-dimensional Reeb flows.} J. Symplectic Geom. 19 (2021), no. 3, 531--566. 

\bibitem{katok} A. Katok, \textit{Lyapunov exponents, entropy and periodic orbits for diffeomorphisms.} Publications math\'ematiques de   l'I.H.\'E.S., tome 51 (1980), 137--173.

\bibitem{Koro2010} A. Koropecki, \textit{Aperiodic invariant continua for surface homeomorphisms.} Math. Z. 266 (2010), no. 1, 229--236.


\bibitem{KLN} A.~Koropecki, P.~ Le Calvez and M.~Nassiri, {\it Prime ends rotation numbers and periodic points.} Duke Math. J. 164 (2015), no. 3, 403 --472. 

\bibitem{LC} P. Le Calvez, {\it Propri\'et\'es dynamiques des diff\'eomorphismes de l'anneau et du tore.}  Ast\'erisque. 204. Paris: Soci\'et\'e Math\'ematique de France, 131 p. (1991). 

\bibitem{LS} P. Le Calvez, M. Sambarino, {\it Homoclinic orbits for area preserving difffeomorphisms of surfaces.} Ergodic Theory Dyn. Syst. 42, No. 3, 1122-1165 (2022). 


\bibitem{Mather} J. N. Mather, \textit{Invariant subsets for area preserving homeomorphisms of surfaces.} Mathematical analysis and applications, Part B, Adv. in Math. Suppl. Stud., vol. 7, Academic Press, New York, 1981, pp. 531--562.

%
%
%
%
%
%
%
%
\bibitem{Zeh} E. Zehnder, {\it Homoclinic points near elliptic fixed points},  Comm. Pure Appl. Math. 26 (1973), 131 -- 182. 


\end{thebibliography}
\end{document}